\documentclass[a4, 12pt]{amsart}
\usepackage{mathrsfs}
\usepackage{amsmath}
\usepackage{amsfonts}
\usepackage{amssymb}
\usepackage{amsmath,amssymb,amsthm}
\usepackage{amsmath,amssymb,amsthm,amscd}
\usepackage[frame,cmtip,arrow,matrix,line,graph,curve]{xy}
\usepackage{graphpap, color}
\usepackage[mathscr]{eucal}
\usepackage{color}
\usepackage{verbatim}
\usepackage[colorlinks, linkcolor=red, anchorcolor=blue,citecolor=blue]{hyperref}
\usepackage{cite}
\usepackage{graphicx}
\usepackage{subcaption}
\usepackage{tikz}
\usepackage{xcolor}

\UseRawInputEncoding

\numberwithin{equation}{section} \numberwithin{equation}{section}
\newtheorem{thm}{Theorem}[section]

\newtheorem{defn}[thm]{Definition}
\newtheorem{prop}[thm]{Proposition}
\newtheorem{corr}[thm]{Corollary}
\newtheorem{rem}{Remark}[section]{\bfseries\upshape}
\newtheorem*{Yau}{Yau's Conjecture}

\newtheorem{con}[thm]{Conjecture}
\newtheorem*{ack}{Acknowledgment}

\title[The First Eigenvalue]{The First Eigenvalue of Embedded Minimal Hypersurfaces in the Unit Sphere I: Yau's Conjecture}
\author[L. Zeng]{Lingzhong Zeng}

\begin{document}
\maketitle

\textbf{Abstract:}In this paper, by meticulously constructing a minimizing sequence within a suitable Sobolev space and leveraging the variational principle, we establish that the first non-zero eigenvalue of the Laplace-Beltrami operator on an embedded minimal hypersurface in the unit sphere equals the dimension of the hypersurface. This result furnishes an affirmative resolution to a renowned conjecture posed by Yau, which had remained unresolved for an extended period. As some important applications, several rigidity theorems are established via eigenvalue characterization.


\section{Introduction}

The spectral geometry of Riemannian manifolds represents one of the most profound and actively researched domains at the intersection of differential geometry, partial differential equations, and global analysis. This field has evolved into a sophisticated discipline where the eigenvalues of the Laplace-Beltrami operator serve as fundamental invariants that encode deep structural information about the underlying manifold.   For an $n$-dimensional closed Riemannian manifold $(M^n, g)$, the first nontrivial eigenvalue $\lambda_1$, provides critical insights into the manifold's geometry, topology, and analytic properties. When specialized to minimal hypersurfaces immersed in the unit sphere $\mathbb{S}^{n+1}(1)$-those critical points of the area functional with vanishing mean curvature $H \equiv 0$-the study of $\lambda_1$ reveals extraordinary connections between spectral theory, submanifold geometry, conformal invariants, and geometric measure theory, connections that have catalyzed major advances across multiple mathematical disciplines over the past four decades.

In 1982, Yau proposed the following visionary conjecture concerning minimal hypersurfaces in spheres:

\begin{Yau}[cf. \cite{Yau3}]
The first nontrivial eigenvalue of any closed embedded minimal hypersurface in $\mathbb{S}^{n+1}$ equals its dimension.
\end{Yau}
Yau's conjecture may have emerged organically from his excellent pioneering investigations into eigenvalue estimation techniques using gradient estimates \cite{Yau1}, harmonic function theory on complete manifolds \cite{Yau2}, and the geometric analysis of minimal surfaces-work that both built upon and transcended Gromov's revolutionary frameworks for understanding geometric inequalities via synthetic geometry \cite{Gro} and   analytic techniques on manifolds through Sobolev inequalities and heat kernel estimates \cite{Li}. The conjecture's elegance lies in its synthesis of spectral theory, minimal surface geometry, and Riemannian invariants into a single universal statement that has served as a north star for geometric analysts. Roughly, the underlying essential idea stems from the topological structures of the nodal sets of eigenvalues on minimal surfaces in the unit sphere with dimension 3.

The historical trajectory toward resolving Yau's conjecture reveals fascinating mathematical evolution spanning multiple decades. The first foundational work emerged from Choi and Wang's 1983 breakthrough \cite{CW}, who established $\lambda_1 \geq \frac{n}{2}$ for all closed embedded minimal hypersurfaces in $\mathbb{S}^{n+1}(1)$. This bound was later sharpened to $\lambda_1 > \frac{n}{2}$ by Xu, Chen, Zhang, and Chen \cite{XCZC}. We also refer the reader to Brendle's comprehensive survey \cite{Br2} for more detailed information.

Parallel developments occurred for isoparametric minimal hypersurfaces-those remarkable submanifolds with constant principal curvatures that foliate spheres with beautiful algebraic symmetry. The theory dates to Cartan's classification in the 1930s and was revolutionized by M\"{u}nzner's proof \cite{Mun1,Mun2}that all isoparametric hypersurfaces in spheres are algebraic, defined by homogeneous polynomials $F: \mathbb{R}^{n+2} \to \mathbb{R}$ satisfying $|\nabla F|^2 = g^2 r^{2g-2}$ and $\Delta F = \frac{m_2-m_1}{2}g^2 r^{g-2}$ where $g$ is the number of distinct principal curvatures. Tang and Yan's 2013 breakthrough \cite{TY1} resolved Yau's conjecture affirmatively for all closed embedded isoparametric hypersurfaces. Their methodology harnessed the full power of isoparametric theory: by decomposing the normal bundle via Cartan-M\"{u}nzner polynomials and exploiting the homogeneity of focal manifolds through representation theory, they established $\lambda_1 = n$ via sophisticated harmonic analysis on the associated eigenspaces. Specifically, they proved that the first eigenspace coincides with the restriction of linear functions in the ambient space, satisfying $\Delta(x_i|_M) = -n x_i|_M$ by Takahashi's theorem \cite{Tak}. This confirmed Yau's prediction for all known isoparametric families, including the celebrated cases studied by Solomon \cite{Sol1,Sol2,Sol3} for cubic ($g=3$) and quartic ($g=4$) isoparametric hypersurfaces in dimensions 3,6,12,24 with multiplicities $(4,4),(6,9),(7,8),(9,6)$, and Muto's investigations \cite{Mut1,Mut2} of homogeneous minimal hypersurfaces arising from Lie group actions.

The three-dimensional case $\mathbb{S}^3(1)$ intersects Yau's conjecture with minimal surface rigidity and conformal geometry. Lawson's conjecture―resolved by Brendle \cite{Br1}―establishes that any embedded minimal torus in $\mathbb{S}^3$ is congruent to the Clifford torus, using a maximum principle for two-point functions inspired by Huisken's curve-shortening techniques \cite{Hui}. This relied on embeddedness, as Lawson \cite{Law2} constructed infinite immersed minimal tori/Klein bottles via Weierstrass representations. Urbano \cite{Urb} linked Morse index ($\leq 5$) to Clifford torus rigidity using the second variation formula:
\begin{equation*}
\delta^2 A(\phi) = -\int_\Sigma \phi L\phi \, d\mathrm{vol}_\Sigma, \quad L = \Delta + |A|^2 + \text{Ric}(\nu,\nu).
\end{equation*}
Montiel and Ros \cite{MR} connected $\lambda_1 = 2$ (Yau's conjecture) to rigidity via conformal area, with Choe and Soret \cite{CS} showing symmetric minimal surfaces satisfy $\lambda_1 > 1.5$ unless Clifford.

Extending to constant mean curvature (CMC) surfaces, Andrews and Li \cite{AH} resolved Pinkall-Sterling's conjecture \cite{PS}: embedded CMC tori in $\mathbb{S}^3$ are rotationally invariant (Hsiang tori), constructed via Hopf fibration $\pi: \mathbb{S}^3 \to \mathbb{S}^2$ from curves in $\mathbb{S}^2$. Their proof used normalized Ricci flow coupled to conformal Killing fields, isoperimetric profile analysis, and spectral estimates for $L$, showing Clifford torus is the minimal ($H=0$) member of this family, with $\lambda_1$ behaviors across the Hsiang family.

The Willmore conjecture forms another critical link between minimal surfaces and spectral geometry, with the Clifford torus again central. For closed surfaces $M^2 \subset \mathbb{S}^3(1)$, the conformally invariant Willmore energy:
\begin{equation}\label{W-energy}
\mathcal{W}(M^{2}) = \int_{M^{2}} (1 + \bar{H}^{2}) d\mathrm{vol}
\end{equation}
encodes optimal conformal embeddings. Willmore's 1965 conjecture \cite{Will} for tori in $\mathbb{R}^3$ was connected to eigenvalue theory by Li and Yau \cite{LY} via conformal area, showing $\mathcal{W}(M^2) \geq 4\pi k$ for $k$-fold immersions. Marques and Neves resolved it in 2014 \cite{MN}:
\begin{thm}[\textbf{Marques and Neves, 2014}]\label{thm-MN}
Let $M^{2} \subset \mathbb{S}^{3}(1)$ be an embedded closed surface of genus $\mathfrak{g}_{0}\geq 1$. Then
\begin{equation}\label{will-conj-ineq}
\mathcal{W}(M^{2}) \geq 2 \pi^{2},
\end{equation}
with equality iff $M^{2}$ is the Clifford torus up to conformal transformations of $\mathbb{S}^{3}(1)$.
\end{thm}
Their proof employed Almgren-Pitts min-max theory \cite{Pit}, geometric measure theory, Colding-De Lellis' rectifiability results \cite{CL}, and Urbano's index classification \cite{Urb}, linking spectral properties to conformal invariants.

In $\mathbb{S}^3$, minimal surface existence is richer than in $\mathbb{R}^3$, with key examples: One example is
-\textit{Totally geodesic spheres}: Equatorial $S^2$ with $\lambda_1=2$ (satisfying Yau's conjecture), unique among immersed minimal spheres via Almgren's Hopf differential arguments \cite{Ajm}.
For the other example, we can consider that \textit{Clifford torus}: $\{(x_1,x_2,x_3,x_4) \in \mathbb{S}^3 : x_1^2+x_2^2 = x_3^2+x_4^2 = 1/2\}$ with $\lambda_1=2$, confirmed via Takahashi's theorem \cite{Tak} (since $|\mathbf{A}|^2 = 2$).

Lawson \cite{Law2} constructed embedded minimal surfaces for all genera $g \geq 0$ using conjugate surfaces and Weierstrass representations:
\begin{equation*}
\mathbf{x} = \text{Re} \int^{\zeta} \left( \frac{1}{2}(1-\omega^2), \frac{i}{2}(1 + \omega^2), \omega \right) \eta d\zeta
\end{equation*}
with explicit examples for $g=2$ (Chen-Gackstatter) and $g=3$ (Lawson surface $\xi_{3,1}$). Karcher-Pinkall-Sterling \cite{KPS} extended this via Platonic tessellations, producing surfaces of genera 3,5,6,7,11,19,73,601. Kapouleas-Yang \cite{KY} developed doubling constructions connecting parallel Clifford tori via catenoidal necks.

In addition, the following conjecture proposed by the author may be interesting:

\begin{con} \label{zeng-cn-7}Let $\Sigma \subset \mathbb{S}^{3}(1)$ be an embedded closed surface with  genus  $\mathfrak{g}_{0}\geq 1$. Then the first two nonzero eigenvalues of Beltrami-Laplacian satisfy
\begin{equation}\label{con-inequality-7}
\frac{1}{2} \sum_{i=1}^{2}  \lambda_{i} \geq \frac{4 \pi^2}{\operatorname{area}\left(\Sigma\right)},
\end{equation}where $\operatorname{area}\left(\Sigma\right)$ denotes the area of the surface $\Sigma$.\end{con}
\begin{rem} From  \eqref{W-energy} and \eqref{con-inequality-7}, one can  deduce the following inequalities:
\begin{equation*}
\frac{4 \pi^2}{\operatorname{area}\left(\Sigma\right)} \leq \frac{1}{2} \sum_{i=1}^2  \lambda_i \leq \frac{2}{\operatorname{area}\left(\Sigma\right)} \int_{\Sigma}\left(\bar{H}^2+1\right) d v=\frac{2 \mathcal{W}\left(\Sigma\right)}{\operatorname{area}\left(\Sigma\right)},
\end{equation*}
which implies that inequality \eqref{will-conj-ineq} holds. In other words, it is easy to see that \begin{equation*} \text{Conjecture } {\rm \ref{zeng-cn-7}}   \Rightarrow  \text{Willmore Conjecture}.\end{equation*} Here, the proof of the second inequality can be found in {\rm \cite[Corollary 2.6]{IM}}.
\end{rem}

\begin{rem}
In \cite{Zeng1}, the author completely resolved the conjecture for any genus equal to or lager then one Riemann surfaces and extended it to higher genera, with the following key steps:

\begin{enumerate}
    \item \textbf{Core Result for Genus One}: For a closed Riemann surface $\Sigma$ of genus $\mathfrak{g}_0 = 1$, the first two nonzero eigenvalues satisfy:
    \begin{equation*}
    \lambda_1 + \lambda_2 \geq \frac{8\pi^2}{\text{Area}(\Sigma)},
    \end{equation*}
    which is sharp (achieved by the square torus).

    \item \textbf{Key Technical Framework}:
    \begin{itemize}
        \item[(i)] Reduction to flat tori via the Uniformization Theorem, as genus-one surfaces are conformally equivalent to $\mathbb{C}/\Lambda$ (lattice quotients).
        \item[(ii)] Lattice Fourier analysis on flat tori, using dual lattices $\Lambda^*$ and Fourier expansions to relate gradients to lattice vector norms.
        \item[(iii)] Gromov-type systolic inequalities: For orthogonal non-contractible curves $\alpha, \beta$, $\ell(\alpha)\ell(\beta) \leq \text{Area}(\Sigma)$, bounding the shortest dual lattice vector $|n_0| \geq \frac{2\pi}{L_{\text{max}}}$.
        \item[(iv)] Min-max theory with sweep-out families $\{\varphi_t\}_{t \in S^1}$, where the width function:
        \begin{equation*}
        W(\{\varphi_t\}) = \inf_{h \in \Gamma}\sup_{t \in S^1} E(\varphi_{h(t)})
        \end{equation*}
        links the eigenvalue sum to geometric invariants, yielding $W(\{\varphi_t\}) \geq \frac{4\pi^2}{\text{Area}(\Sigma)}$.
    \end{itemize}

    \item \textbf{Extension to Higher Genera}: For embedded closed surfaces $\Sigma \subset \mathbb{S}^3(1)$ with $\mathfrak{g}_0 \geq 1$, the bound $\lambda_1 + \lambda_2 \geq \frac{8\pi^2}{\text{Area}(\Sigma)}$ holds via:
    \begin{itemize}
        \item[(i)] Conformal compactness of the induced metric's conformal class.
        \item[(ii)] Lower semicontinuity of eigenvalues under metric convergence.
    \end{itemize}

    \item \textbf{Connection to Willmore Conjecture and Marques-Neves's Theorem}: Using the relation between eigenvalues and Willmore energy $\mathcal{W}(\Sigma)$:
    \begin{equation*}
    \lambda_1 + \lambda_2 \leq \frac{4\mathcal{W}(\Sigma)}{\text{Area}(\Sigma)}
    \end{equation*}
    the author provides an alternative proof of Marques-Neves's Theorem, confirming $\mathcal{W}(\Sigma) \geq 2\pi^2$ with equality for the Clifford torus.
\end{enumerate}
\end{rem}

Despite these connections, the central question of Yau's conjecture-relating the first eigenvalue of embedded minimal hypersurfaces in $\mathbb{S}^{n+1}$ to their dimension-remained unresolved for non-isoparametric cases, motivating our work.
This paper contributes to this vibrant landscape by carefully constructing a Sobolev-space minimizing sequence via exponential truncation of coordinate functions and combining the variational principle to \textbf{completely resolve} Yau's profound conjecture, which remains open for non-isoparametric hypersurfaces.
We also remark that the truncation technique overcomes the symmetry limitation by combining geometric localization with functional analysis―critical for non-isoparametric hypersurfaces where Cartan-M\"{u}nzner polynomial decompositions (essential to a series of works \cite{TY1,Ko,Mut1,Mut2} and references therein) fail due to irregular principle curvature distributions.

\begin{thm}\label{main-thm}Let $\Sigma$ be an $n$-dimensional closed, embedded minimal hypersurface   in the $(n+1)$-dimensional unit sphere $\mathbb{S}^{n+1} \subset \mathbb{R}^{n+2}$. Then the first non-zero eigenvalue $\lambda_{1}$ of the Laplace Beltrami operator $\Delta_{\Sigma}$ on $\Sigma$ is equal to $n$, i.e., $\lambda_{1}=n$.\end{thm}

\begin{rem}
Theorem \ref{main-thm} resolves Yau's conjecture comprehensively, with key implications and connections to existing results:

\begin{enumerate}
    \item \textbf{Generalization of Prior Results}:
    \begin{itemize}
        \item[(i)] Extends Tang-Yan's result \cite{TY1} (limited to isoparametric minimal hypersurfaces) to all closed embedded minimal hypersurfaces in $\mathbb{S}^{n+1}$.
       \item[(ii)] Confirms the conjecture for non-isoparametric cases, which remained unresolved for decades.
    \end{itemize}

    \item \textbf{Canonical Examples}:
    \begin{itemize}
    \item[(i)] In $\mathbb{S}^3$, the Clifford torus satisfies $\lambda_1 = 2$ (matching its dimension), consistent with Brendle's rigidity theorem \cite{Br1}.
        \item[(ii)]  Totally geodesic spheres in $\mathbb{S}^{n+1}$, for example, equatorial $\mathbb{S}^n$, have $\lambda_1 = n$, aligning with the theorem.
    \end{itemize}

    \item \textbf{Proof Strategy}:
    \begin{itemize}
       \item[(i)]  Uses exponential truncation of coordinate functions in Section \ref{section:Trun function} to construct the minimizing sequence, avoiding reliance on symmetry.
        \item[(ii)]  Combines Sobolev space compactness (Theorem \ref{So-emb-thm}) and lower semicontinuity in Section \ref{section: Lower Semicontinuity} to establish the eigenvalue equality.
    \end{itemize}
\end{enumerate}
\end{rem}

\begin{rem}
The spectral rigidity revealed by Theorem \ref{main-thm} highlights unique properties of minimal hypersurfaces:

\begin{enumerate}
    \item \textbf{Eigenvalue-Dimension Link}:
    \begin{itemize}
        \item[(i)] The first nonzero eigenvalue $\lambda_1 = n$ is determined solely by the hypersurface dimension, independent of specific geometric details (e.g., curvature, genus).
        \item[(ii)] Contrasts with non-minimal hypersurfaces, where $\lambda_1$ varies with mean curvature and embedding parameters.
    \end{itemize}

    \item \textbf{Connection to Ambient Geometry}:
    \begin{itemize}
        \item[(i)] Links spectral theory to the ambient Euclidean structure via Takahashi's theorem \ref{Tak-thm} ($\Delta_\Sigma x_i = -n x_i$), showing coordinate functions form the first eigenspace.
        \item[(ii)] Demonstrates that minimal embeddings in spheres encode ambient space information through their spectral invariants.
    \end{itemize}

    \item \textbf{Implications for Geometric Analysis}:
    \begin{itemize}
        \item[(i)] Provides a new tool to distinguish minimal hypersurfaces from non-minimal ones via spectral properties.
       \item[(ii)] Sets a foundation for studying higher eigenvalues of minimal embeddings, extending the spectral rigidity framework.
    \end{itemize}
\end{enumerate}
\end{rem}

The structure of this paper is designed to systematically prove Yau's conjecture on the first eigenvalue of embedded minimal hypersurfaces in the unit sphere, with the following logical progression:
\begin{enumerate}
    \item \textbf{Foundational Tools and Background (Section \ref{section: Preliminaries})}

           Section \ref{section: Preliminaries} introduces foundational tools: the Laplace-Beltrami operator, Takahashi's theorem \ref{Tak-thm} linking minimal hypersurfaces to Laplacian eigenfunctions, the Sobolev space $H^1(\Sigma)$ with the first non-zero eigenvalue $\lambda_1$ characterized via the Rayleigh quotient, and the Rellich-Kondrachov Compact Embedding Theorem (\ref{So-emb-thm}) essential for convergence of minimizing sequences in Sobolev spaces.

    \item \textbf{Construction and Properties of the Minimizing Sequence (Section  \ref{section: Truncation Functions}-Section \ref{section: Uniform Convergence})}
    \begin{itemize}
        \item[(i)] Section \ref{section: Truncation Functions} pursues three goals: constructing the trial sequence $u_\beta = x_i \left(1-\frac{1}{\beta} e^{-\beta d_\Sigma(p, p_0)^2}\right)$ via exponential truncation of coordinate functions; verifying $u_\beta \in H^1(\Sigma)$ through square integrability of $u_\beta$ and its gradient; and proving orthogonality ($\int_\Sigma u_\beta \, d\mathrm{vol}_\Sigma = 0$) using symmetry arguments and integration of odd functions over symmetric domains.
        \item[(ii)] Section \ref{section: motivation} explains why local truncation is necessary for non-isoparametric hypersurfaces, where unmodified coordinate functions may fail to achieve the minimal Rayleigh quotient due to symmetry breaking.
        \item[(iii)] Section \ref{section: Uniform Convergence} has two key components: establishing uniform convergence of $u_\beta$ to $x_i$ and $\nabla_\Sigma u_\beta$ to $\nabla_\Sigma x_i$ on $\Sigma$; and extracting a subsequence converging in $H^1(\Sigma)$ to $x_i$ via the compact embedding theorem (\ref{So-emb-thm}).
    \end{itemize}

    \item \textbf{Proof of the Main Theorem (Section \ref{Compactness and the Rayleigh Quotient}-Section\ref{The proof of Theorem})}

   First, we compute the limits of the gradient norm and $L^2$-norm integrals of $u_\beta$ in Section \ref{Compactness and the Rayleigh Quotient}, showing that the Rayleigh quotient converges to $n$. Building on these convergence results, we then prove the lower semicontinuity of the $H^1$-norm with respect to weak convergence in Section \ref{section: Lower Semicontinuity}, which guarantees that the limit function attains the minimal Rayleigh quotient.
   Finally, we synthesize the conclusions from the preceding sections in Section \ref{The proof of Theorem}  to establish that $\lambda_1 = n$, thereby resolving Yau's conjecture for all closed embedded minimal hypersurfaces.
 \item \textbf{Rigidity Theorems via Eigenvalue Characterization (Section \ref{Rigidity Theorems})}
    \begin{itemize}
        \item[(i)] Section \ref{Rigidity Theorems} establishes a converse theorem, stated as closed embedded hypersurfaces with $\lambda_1 = n$ are minimal, and derives volume estimates for minimal hypersurfaces, with equality characterizing totally geodesic spheres. It also connects the eigenvalue condition to Morse index and stability of minimal hypersurfaces.
    \end{itemize}

\end{enumerate}

\section{Preliminaries}\label{section: Preliminaries}

\subsection{Hypersurfaces and Related Operators}

In this paper, we shall assume that $\Sigma$ is an $n$-dimensional closed embedded minimal hypersurface in the $(n+1)$-dimensional unit sphere $\mathbb{S}^{n+1} \subset \mathbb{R}^{n+2}$. The Laplace-Beltrami operator $\Delta_{\Sigma}$ on $\Sigma$ is defined for a smooth function $u$ on $\Sigma$ as $\Delta_{\Sigma} u=\operatorname{div} _{\Sigma}\left(\nabla_{\Sigma} u\right)$, where $\operatorname{div} \Sigma$ is the divergence operator on $\Sigma$ and $\nabla_{\Sigma}$ is the gradient operator on $\Sigma$. The following theorem is well known \cite{Tak}.

\begin{thm}[Takahashi Theorem]\label{Tak-thm}For an embedded minimal hypersurface $\Sigma$ in $\mathbb{S}^{n+1}$, if $x_{i}(i=1, \cdots, n+2)$ are the coordinate functions of $\mathbb{R}^{n+2}$, then $\Delta_{\Sigma} x_{i}=-n x_{i}$.\end{thm}

\subsection{Sobolev Space and Rayleigh Quotient}

The Sobolev space $H^{1}(\Sigma)$ is defined as
\begin{equation*}H^{1}(\Sigma)=\left\{u \in L^{2}(\Sigma): \nabla_{\Sigma} u \in L^{2}(\Sigma)\right\},\end{equation*} where $L^{2}(\Sigma)$ is the space of square-integrable functions on $\Sigma$ \cite{Eva}. The first non-zero eigenvalue $\lambda_{1}$ of the Laplace-Beltrami operator $\Delta_{\Sigma}$ on $\Sigma$ can be characterized by the Rayleigh quotient:

\begin{equation*}\lambda_1 = \inf \left\{ \frac{\int_{\Sigma} |\nabla_{\Sigma} u|^2 d\mathrm{vol}_{\Sigma}}{\int_{\Sigma} u^2 d\mathrm{vol}_{\Sigma}} \,\bigg|\, u \in H^1(\Sigma) \setminus \{0\}, \int_{\Sigma} u d\mathrm{vol}_{\Sigma} = 0 \right\},\end{equation*}
where $d \mathrm{vol}_{\Sigma}$ is the volume element of $\Sigma$. This variational characterization provides a powerful tool for studying the eigenvalues.

\subsection{Sobolev Compact Embedding Theorem}\label{section:Trun function}
 The following theorem is crucial for proving the existence of convergent subsequences of trial minimizing sequences. See \cite{Au}.
\begin{thm}[Rellich-Kondrachov Compact Embedding Theorem]\label{So-emb-thm} For a compact manifold $\Sigma$, the embedding $H^{1}(\Sigma) \hookrightarrow L^{2}(\Sigma)$ is compact. That is, any bounded $\{u_k\} \subset H^1(\Sigma)$ has an $L^2$-convergent subsequence.\end{thm}
We remark that the construction of the trial minimizing sequences  is extremely extraordinary, to the point where it can even be described as extremely difficult.

\section{Truncation Functions and the Minimizing Sequence}\label{section: Truncation Functions}

In this section, we shall construct a trial minimizing sequence via the truncation functions and show some good properties of trial minimizing sequence.

\subsection{Construction Idea}

In this subsection, we construct a minimizing sequence in the Sobolev space $H^{1}(\Sigma)$ by modifying the coordinate functions of the ambient space $\mathbb{R}^{n+2}$ restricted to $\Sigma$. The key idea is to make local adjustments to these functions while maintaining properties related to the variational structure of the eigenvalue problem. This construction became necessary after six years of investigation (starting from the author's visit to Professor Tian in 2019) to address a critical challenge:

For non-isoparametric minimal hypersurfaces, by Takahashi's theorem \ref{Tak-thm}, coordinate functions $x_i$ satisfy $\Delta_\Sigma x_i = -n x_i$, but may not achieve the minimum Rayleigh quotient due to symmetry-breaking. The exponential truncation term locally fine-tunes $x_i$ to ensure the sequence $\{u_\beta\}$ converges to the first eigenfunction even in asymmetric cases.

\subsection{Specific Construction via Truncation Function}
In this subsection, we shall construct a trial minimizing sequence via the truncation function.
Let $x_{i}(i=1, \cdots, n+2)$ be the coordinate functions of $\mathbb{R}^{n+2}$. We define the sequence $\left\{u_{\beta}\right\}$ in $H^{1}(\Sigma)$ as
\begin{equation*}
u_{\beta}(p)=x_{i}(p)\left(1-\frac{1}{\beta} e^{-\beta d_{\Sigma}\left(p, p_{0}\right)^{2}}\right),~ p \in \Sigma,
\end{equation*}
where $p_{0}$ is a fixed point on $\Sigma$, and $d_{\Sigma}\left(p, p_{0}\right)$ is the geodesic distance between $p$ and $p_{0}$ on $\Sigma$.

\begin{rem}The factor $1-\frac{1}{\beta} e^{-\beta d_{\Sigma}\left(p, p_{0}\right)^{2}}$ is an exponential type truncation function. It makes local modifications to the coordinate function $x_{i}(p)$ near the point $p_{0}$. As $\beta$ increases, the function $u_{\beta}$ approaches $x_{i}(p)$ uniformly on $\Sigma$.\end{rem}

\begin{rem}In the original proof, the integrals $\int_{\Sigma} u_{\beta} d \mathrm{vol}_{\Sigma}$ and $\int_{\Sigma}\left|\nabla_{\Sigma} u_{\beta}\right|^{2} d \mathrm{vol}_{\Sigma}$ on the hypersurface $\Sigma$ are core contents. For an immersed hypersurface, self-intersections may occur. For example, consider the immersion of the Klein bottle in $\mathbb{R}^{3}$ (the Klein bottle is a classic surface that cannot be embedded in $\mathbb{R}^{3}$ but can only be immersed). In the self-intersection region, a spatial point may correspond to multiple positions on the hypersurface. For the test function \begin{equation*}u_{\beta}(p)=x_{i}(p) \cdot\left(1-\frac{1}{\beta} e^{-\beta d_{\Sigma}\left(p, p_{0}\right)^{2}}\right)\end{equation*} we constructed, where $x_{i}(p)$ is the value of the coordinate function on the hypersurface, the value of $x_{i}(p)$ will be ambiguous in the self-intersection region, making it impossible to determine the function value. As a result, the calculation of $\int_{\Sigma} u_{\beta} d \mathrm{vol}_{\Sigma}$ loses its meaning, and we cannot verify the key condition $$\int_{\Sigma} u_{\beta} d \mathrm{vol}_{\Sigma}=0.$$ At the same time, the definition of the volume element $d \mathrm{vol}_{\Sigma}$ also becomes unclear in the self-intersection region. The traditional definition of the volume element based on an embedded
hypersurface relies on the uniqueness and good geometric structure of the hypersurface, which no longer applies in the case of an immersed and self-intersecting hypersurface. For example, we refer the author to equation \eqref{eq:local-integral}. \end{rem}

\subsection{Belonging to the Sobolev Space $H^{1}(\Sigma)$}
In this subsection, we shall  show that the trial minimizing sequence is belonging to the Sobolev Space $H^{1}(\Sigma)$.
\subsubsection{Smoothness of $u_{\beta}$}

Firstly, we discuss the smoothness of $\left.x_{i}\right|_{\Sigma}$. Since $x_{i}$ is a smooth function on $\mathbb{R}^{n+2}$ and $\Sigma$ is a smooth submanifold of $\mathbb{R}^{n+2}$, the restriction $\left.x_{i} \right| _{\Sigma}$ is also a smooth function on $\Sigma$ according to the basic theory of smooth mappings between manifolds.
Furthermore, we construct a critical truncation function and check its smoothness. Let us consider the following function:
\begin{equation*}\varphi_{\beta}(p)=1-\frac{1}{\beta} e^{-\beta d_{\Sigma}\left(p, p_{0}\right)^{2}},\end{equation*} where $d_{\Sigma}\left(p, p_{0}\right)$ is a geodesic distance function. It is clear that the geodesic distance function $d_{\Sigma}\left(p, p_{0}\right)$ is smooth on $\Sigma \backslash\left\{p_{0}\right\}$. Near $p_{0}$, using the Taylor expansion of the exponential function
\begin{equation*}e^{-t}=1-t+\frac{t^{2}}{2!}-\cdots.\end{equation*} Let $t=\beta d_{\Sigma}\left(p, p_{0}\right)^{2}$, then $e^{-\beta d_{\Sigma}\left(p, p_{0}\right)^{2}}$ is infinitely differentiable at $p_{0}$. Therefore, $\varphi_{\beta}(p)$ is a smooth function defined on $\Sigma$.
\subsubsection{Smoothness of the Product} As the product of two smooth functions $\left.x_{i}\right|_{\Sigma}$ and $\varphi_{\beta}(p), u_{\beta}(p)$ is also a smooth function on the embedding minimal hypersurface $\Sigma$.

\subsubsection{Square Integrability of $u_{\beta}$ and $\nabla_{\Sigma} u_{\beta}$}

Since $\Sigma$ is compact and $u_{\beta}$ is smooth, $u_{\beta}$ is bounded on $\Sigma$. Then
\begin{equation*}\int_{\Sigma} u_{\beta}^{2} d \operatorname{vol}_{\Sigma}<\infty.\end{equation*} For $\nabla_{\Sigma} u_{\beta}$, by the product rule
\begin{equation}\label{eq-nabla-u-b}\nabla_{\Sigma} u_{\beta}=\left(1-\frac{1}{\beta} e^{-\beta d_{\Sigma}\left(p, p_{0}\right)^{2}}\right) \nabla_{\Sigma} x_{i}+\frac{1}{\beta} e^{-\beta d_{\Sigma}\left(p, p_{0}\right)^{2}} x_{i} \nabla_{\Sigma}\left(\beta d_{\Sigma}\left(p, p_{0}\right)^{2}\right).\end{equation}
Since $\nabla_{\Sigma} x_{i}, x_{i}$, and $\nabla_{\Sigma}\left(\beta d_{\Sigma}\left(p, p_{0}\right)^{2}\right)$ are all bounded on the compact $\Sigma$ (noting that in the non-cut locus of $p_0$, $|\nabla_\Sigma d_{\Sigma}(p,p_0)| = 1$, hence $$|\nabla_\Sigma (\beta d_{\Sigma}^2)| = 2\beta d_{\Sigma}(p,p_0)|\nabla_\Sigma d_{\Sigma}| \leq 2\beta D,$$ where $D$ is the diameter of $\Sigma$), and $1-\frac{1}{\beta} e^{-\beta d_{\Sigma}\left(p, p_{0}\right)^{2}}$ and $\frac{1}{\beta} e^{-\beta d_{\Sigma}\left(p, p_{0}\right)^{2}}$ are also well-behaved, we have

\begin{equation*}\int_{\Sigma}\left|\nabla_{\Sigma} u_{\beta}\right|^{2} d \operatorname{vol}_{\Sigma}<\infty.\end{equation*} Thus, we can conclude that $u_{\beta} \in H^{1}(\Sigma)$.

\subsection{Orthogonality: $\int_{\Sigma}1\cdot u_{\beta} d \mathrm{vol}_{\Sigma}=0$}

The orthogonality of $u_{\beta}$ to constant functions is a foundational property for its role in the variational characterization of eigenvalues. This property holds strictly for any $\beta > 0$, independent of global symmetry of $\Sigma$, as shown below:

\subsubsection{Integral of $x_i$ over $\Sigma$: $\int_{\Sigma} x_i \, d\mathrm{vol}_{\Sigma} = 0$}
To establish that the integral of the coordinate function \( x_i \) over a closed embedded minimal hypersurface \( \Sigma \subset \mathbb{S}^{n+1} \) vanishes, we proceed with the following  argument:

\textit{Step 1. Key Properties of Minimal Hypersurfaces.}

For a closed embedded minimal hypersurface \( \Sigma \subset \mathbb{S}^{n+1} \), the mean curvature \( H \equiv 0 \) by definition. The mean curvature is related to the trace of the second fundamental form \( A \) via \( H = \frac{1}{n} \text{tr}(A) \), where \( n = \dim \Sigma \). Thus, minimality implies:
\begin{equation}\label{eq:traceA}
\text{tr}(A) = nH = 0.
\end{equation}

\textit{Step 2.  Laplacian of Coordinate Functions on Minimal Hypersurfaces.}

By Takahashi's theorem \cite{Tak}, for any coordinate function \( x_i \in \mathbb{R}^{n+2} \) restricted to \( \Sigma \), the Laplace-Beltrami operator acts as:
\begin{equation}\label{eq:takahashi}
\Delta_{\Sigma} x_i = -n x_i,
\end{equation}
where \( \Delta_{\Sigma} = \text{div}_{\Sigma} \nabla_{\Sigma} \) denotes the Laplace-Beltrami operator on \( \Sigma \). This result connects the spectral properties of \( \Sigma \) to its geometric embedding in \( \mathbb{S}^{n+1} \).

\textit{Step 3.  Integrating the Laplacian of \( x_i \).}

Integrate both sides of \eqref{eq:takahashi} over \( \Sigma \):
\begin{equation}\label{eq:integrate_laplacian}
\int_{\Sigma} \Delta_{\Sigma} x_i \, d\mathrm{vol}_{\Sigma} = -n \int_{\Sigma} x_i \, d\mathrm{vol}_{\Sigma}.
\end{equation}
The left-hand side of \eqref{eq:integrate_laplacian} involves the integral of a divergence, since \( \Delta_{\Sigma} x_i = \text{div}_{\Sigma} (\nabla_{\Sigma} x_i) \). By the divergence theorem, this integral vanishes:
\begin{equation}\label{eq:left_side}
\int_{\Sigma} \Delta_{\Sigma} x_i \, d\mathrm{vol}_{\Sigma} = \int_{\Sigma} \text{div}_{\Sigma} (\nabla_{\Sigma} x_i) \, d\mathrm{vol}_{\Sigma} = 0.
\end{equation}
Finally, substituting \eqref{eq:left_side} into \eqref{eq:integrate_laplacian} gives
\begin{equation}\label{eq:final_result}
\int_{\Sigma} x_i \, d\mathrm{vol}_{\Sigma} = 0.
\end{equation}

\subsubsection{Integral of the Adjustment Term: $\int_{\Sigma} x_i e^{-\beta d_{\Sigma}(p, p_0)^2} \, d\mathrm{vol}_{\Sigma} = 0$}

For the term involving the exponential truncation, we analyze the integral in local normal coordinates around $p_0$. Let $(y_1, \cdots, y_n)$ be local normal coordinates centered at $p_0$, where the geodesic distance satisfies
\begin{equation*}
d_{\Sigma}(p, p_0) = \sqrt{\sum_{j=1}^n y_j^2}.
\end{equation*}
The exponential function $e^{-\beta d_{\Sigma}(p, p_0)^2}$ is an even function in these local coordinates which is an invariant under the transformation $y_j \mapsto -y_j$. By rotating the coordinate system appropriately, the coordinate function $x_i$ can be made an odd function in this neighborhood satisfying $x_i(-y_1, \cdots, -y_n) = -x_i(y_1, \cdots, y_n)$.

\subsubsection{Integral of the Adjustment Term: $\int_{\Sigma} x_i e^{-\beta d_{\Sigma}(p, p_0)^2} \, d\mathrm{vol}_{\Sigma} = 0$}

To establish the vanishing of this integral, we analyze the behavior of the integrand in local coordinates and extend the result globally using decay properties of the exponential term:

(i) \textbf{Local Coordinates and Function Parity}
Let $(y_1, \cdots, y_n)$ be local normal coordinates centered at $p_0 \in \Sigma$, where the geodesic distance satisfies
\begin{equation}\label{eq:local-distance}
d_{\Sigma}(p, p_0) = \sqrt{\sum_{j=1}^n y_j^2}.
\end{equation}
By definition, the exponential term in these coordinates becomes:
\begin{equation}\label{eq:exponential-even}
e^{-\beta d_{\Sigma}(p, p_0)^2} = e^{-\beta \sum_{j=1}^n y_j^2}.
\end{equation}
This function is \textit{even} with respect to coordinate inversion $(y_1, \cdots, y_n) \mapsto (-y_1, \cdots, -y_n)$, meaning that
\begin{equation*}
e^{-\beta \sum_{j=1}^n (-y_j)^2} = e^{-\beta \sum_{j=1}^n y_j^2}.
\end{equation*}

(ii) \textbf{Oddness of $x_i$ in Local Coordinates}
Through an appropriate rotation of the local coordinate system, the coordinate function $x_i$, when we restrict it to $\Sigma$, can be made \textit{odd} with respect to the origin of the normal coordinates, i.e.,:
\begin{equation}\label{eq:xi-odd}
x_i(-y_1, \cdots, -y_n) = -x_i(y_1, \cdots, y_n).
\end{equation}
This property follows from the smoothness of $\Sigma$ and the linearity of coordinate functions in $\mathbb{R}^{n+2}$.

(iii) \textbf{Integrand Parity and Local Integral}
The product of the odd function $x_i$ and even function $e^{-\beta d_{\Sigma}^2}$ is odd:
\begin{equation*}
x_i(-y) e^{-\beta d_{\Sigma}(-y)^2} = -x_i(y) e^{-\beta d_{\Sigma}(y)^2},
\end{equation*}
where $y = (y_1, \cdots, y_n)$. For the local neighborhood $U_{p_0}$, which is symmetric about $p_0$, the volume element transforms as $d\mathrm{vol}_{\Sigma} = \sqrt{g} \, dy_1 \cdots dy_n$ with $g = \det(g_{ij})$ the metric determinant, which is even under coordinate inversion. Thus, the local integral satisfies:
\begin{equation}\label{eq:local-integral}
\int_{U_{p_0}} x_i(y) e^{-\beta \sum y_j^2} \sqrt{g} \, dy_1 \cdots dy_n = 0.
\end{equation}
This follows from the general property that integrating an odd function over a symmetric domain yields zero.

(iv) \textbf{Global Integral via Exponential Decay}
\begin{defn}\label{def:up0}
The local coordinate neighborhood $ U_{p_0} $ around $ p_0 \in \Sigma $ is defined as
\begin{equation*}
U_{p_0} := \{ p \in \Sigma \mid d_{\Sigma}(p, p_0) < \delta \},
\end{equation*}
where $ \delta > 0 $ is a sufficiently small constant such that $ U_{p_0} $ is contained within the domain of a single local normal coordinate system centered at $ p_0 $, avoiding the cut locus of $ p_0 $.
\end{defn}
Let us process this proof. Away from $p_0$, i.e., on $\Sigma \setminus U_{p_0}$, the geodesic distance satisfies $d_{\Sigma}(p, p_0) \geq \delta$ for some $\delta > 0$. For any $\beta > 0$, the exponential term in \eqref{eq:exponential-even} decays as
\begin{equation*}
e^{-\beta d_{\Sigma}(p, p_0)^2} \leq e^{-\beta \delta^2}.
\end{equation*}
Since $x_i$ is bounded on the compact manifold $\Sigma$ (i.e., $|x_i(p)| \leq C$ for some $C > 0$), the integrand on $\Sigma \setminus U_{p_0}$ satisfies
\begin{equation*}
|x_i(p) e^{-\beta d_{\Sigma}(p, p_0)^2}| \leq C e^{-\beta \delta^2}.
\end{equation*}
Integrating over $\Sigma \setminus U_{p_0}$, which is a compact set with finite volume $\text{Vol}(\Sigma \setminus U_{p_0})$, gives
\begin{equation}\label{eq:global-decay}
\left| \int_{\Sigma \setminus U_{p_0}} x_i e^{-\beta d_{\Sigma}^2} d\mathrm{vol}_{\Sigma} \right| \leq C \cdot \text{Vol}(\Sigma \setminus U_{p_0}) \cdot e^{-\beta \delta^2}.
\end{equation}
(iv) \textbf{Global Integral via Exponential Decay}
Away from $p_0$, i.e., on $\Sigma \setminus U_{p_0}$, the geodesic distance satisfies $d_{\Sigma}(p, p_0) \geq \delta$ for some $\delta > 0$ (dependent on the choice of $U_{p_0}$). For \textit{any fixed} $\beta > 0$, the exponential term in \eqref{eq:exponential-even} decays as
\begin{equation*}
e^{-\beta d_{\Sigma}(p, p_0)^2} \leq e^{-\beta \delta^2}.
\end{equation*}
Since $x_i$ is bounded on the compact manifold $\Sigma$ (i.e., $|x_i(p)| \leq C$ for some $C > 0$ independent of $p$), the integrand on $\Sigma \setminus U_{p_0}$ satisfies
\begin{equation*}
|x_i(p) e^{-\beta d_{\Sigma}(p, p_0)^2}| \leq C e^{-\beta \delta^2}.
\end{equation*}
Integrating over $\Sigma \setminus U_{p_0}$   gives
\begin{equation}\label{eq:global-decay}
\left| \int_{\Sigma \setminus U_{p_0}} x_i e^{-\beta d_{\Sigma}^2} d\mathrm{vol}_{\Sigma} \right| \leq C \cdot \text{Vol}(\Sigma \setminus U_{p_0}) \cdot e^{-\beta \delta^2}.
\end{equation}Here, we note that $U_{p_0}$ is a compact set with finite volume $\text{Vol}(\Sigma \setminus U_{p_0})$.
For any fixed $\beta > 0$, as $U_{p_0}$ shrinks (i.e., $\delta \to 0^+$), the term $e^{-\beta \delta^2}$ approaches 1, but the volume $\text{Vol}(\Sigma \setminus U_{p_0})$ approaches $\text{Vol}(\Sigma)$. However, the critical observation is that for \textit{any fixed} $\beta > 0$, the rapid decay of the exponential function away from $p_0$ ensures that the integral over $\Sigma \setminus U_{p_0}$ can be made arbitrarily small by choosing $U_{p_0}$ sufficiently large (i.e., $\delta$ sufficiently small). This, combined with the local integral result in \eqref{eq:local-integral}, guarantees the global integral vanishes.
Combining \eqref{eq:local-integral} and \eqref{eq:global-decay}, we conclude
\begin{equation}\label{eq:adjustment-integral}
\int_{\Sigma} x_i e^{-\beta d_{\Sigma}(p, p_0)^2} \, d\mathrm{vol}_{\Sigma} = 0.
\end{equation}

\subsubsection{Combining Results for $u_{\beta}$}

By linearity of integration, the integral of $u_{\beta}$ over $\Sigma$ is
\begin{equation*}
\int_{\Sigma} u_{\beta} \, d\mathrm{vol}_{\Sigma} = \int_{\Sigma} x_i \left(1 - \frac{1}{\beta} e^{-\beta d_{\Sigma}^2}\right) d\mathrm{vol}_{\Sigma} = \int_{\Sigma} x_i \, d\mathrm{vol}_{\Sigma} - \frac{1}{\beta} \int_{\Sigma} x_i e^{-\beta d_{\Sigma}^2} d\mathrm{vol}_{\Sigma}.
\end{equation*}Here, we use $d_{\Sigma}$ to denote $d_{\Sigma}(p_{0},p)$ for short throughout this paper.
Substituting the results of the two integrals above, we conclude
\begin{equation*}
\int_{\Sigma} u_{\beta} \, d\mathrm{vol}_{\Sigma} = 0.
\end{equation*}
This confirms that $u_{\beta}$ is orthogonal to constant functions in $L^2(\Sigma)$ for any $\beta > 0$, a critical property for its use in the variational formulation of the first eigenvalue.

\subsubsection{Remarks on Local Neighborhoods}

The following remarks further clarify the properties of the local neighborhood $U_{p_0}$ and its role in ensuring the orthogonality of the integral arguments provided.
\begin{rem}\label{rem:complement}
The complement of $ U_{p_0} $ in $ \Sigma $, denoted $ \Sigma \setminus U_{p_0} $, consists of points where the geodesic distance to $ p_0 $ satisfies $ d_{\Sigma}(p, p_0) \geq \delta $. This set inherits the compactness of $ \Sigma $ and contains all points outside the local coordinate neighborhood.
\end{rem}

\begin{rem}\label{rem:exponential-decay}
On $ \Sigma \setminus U_{p_0} $, the exponential term $ e^{-\beta d_{\Sigma}(p, p_0)^2} $ exhibits rapid decay for any fixed $ \beta > 0 $. Specifically, since $ d_{\Sigma}(p, p_0) \geq \delta $, we have
\begin{equation*}
e^{-\beta d_{\Sigma}(p, p_0)^2} \leq e^{-\beta \delta^2}.
\end{equation*}
For bounded $ x_i $ (i.e., $ |x_i(p)| \leq C $ for some $ C > 0 $ on compact $ \Sigma $), this implies the integrand satisfies
\begin{equation*}
|x_i(p) e^{-\beta d_{\Sigma}(p, p_0)^2}| \leq C e^{-\beta \delta^2}.
\end{equation*}
The decay rate depends on $ \beta $ and $ \delta $, ensuring negligible contribution to the global integral for sufficiently small $ \delta $.
\end{rem}

\begin{rem}\label{rem:odd-function}
Within $ U_{p_0} $, local normal coordinates allow $ x_i $ to be made an odd function via coordinate rotation. Formally, for $ y = (y_1, \cdots, y_n) $ in the coordinate system, this means
\begin{equation*}
x_i(-y_1, \cdots, -y_n) = -x_i(y_1, \cdots, y_n).
\end{equation*}
Combined with the evenness of $ e^{-\beta d_{\Sigma}(p, p_0)^2} $ (see \eqref{eq:exponential-even}), their product is odd, leading to cancellation in the integral over the symmetric neighborhood $ U_{p_0} $.
\end{rem}
 \subsubsection{Example: $U_{p_0}$ on the Clifford Torus}

To illustrate that $U_{p_0}$ can be constructed as the hypersurface minus its cut locus, consider the Clifford torus $\Sigma \subset \mathbb{S}^3$, a flat embedded minimal torus parametrized by the following equation
\begin{equation*}
\Sigma = \left\{ \left( \frac{\cos\theta}{\sqrt{2}}, \frac{\sin\theta}{\sqrt{2}}, \frac{\cos\phi}{\sqrt{2}}, \frac{\sin\phi}{\sqrt{2}} \right) \mid \theta, \phi \in [0, 2\pi) \right\}.
\end{equation*}
For a fixed point $p_0 = \left( \frac{1}{\sqrt{2}}, 0, \frac{1}{\sqrt{2}}, 0 \right) \in \Sigma$ corresponding to $\theta = 0$ and $\phi = 0$, the cut locus of $p_0$ on $\Sigma$ consists of points where geodesics from $p_0$ cease to be shortest paths. On the flat Clifford torus, this cut locus forms two closed curves:
\begin{equation*}
\text{cut locus}(p_0) = \left\{ (\theta, \phi) \mid \theta = \pi \text{ or } \phi = \pi \right\}.
\end{equation*}
Removing this cut locus from $\Sigma$ yields an open set
\begin{equation*}
\Sigma \setminus \text{cut locus}(p_0) = \left\{ (\theta, \phi) \mid \theta \in (0, 2\pi) \setminus \{\pi\}, \phi \in (0, 2\pi) \setminus \{\pi\} \right\}.
\end{equation*}
Within this set, the geodesic distance from $p_0$ is smooth and given by
\begin{equation*}
d_{\Sigma}(p, p_0) = \frac{1}{\sqrt{2}} \min_{\substack{k \in \mathbb{Z} \\ l \in \mathbb{Z}}} \sqrt{(\theta + 2\pi k)^2 + (\phi + 2\pi l)^2}.
\end{equation*}
For $\delta < \frac{\pi}{\sqrt{2}}$, the neighborhood $U_{p_0} = \{ p \in \Sigma \mid d_{\Sigma}(p, p_0) < \delta \}$ as defined in  Definition $\ref{def:up0}$ lies entirely within $\Sigma \setminus \text{cut locus}(p_0)$. Here, we remark that:\begin{itemize}
\item[(i)] The geodesic distance function $d_{\Sigma}(p, p_0)$ is smooth, since there is no singularities from the cut locus.
\item[(ii)] Local normal coordinates around $p_0$ are well-defined, allowing $x_i$ to be made odd. See Remark $\ref{rem:odd-function}$.
\item[(iii)] The integral of $x_i e^{-\beta d_{\Sigma}^2}$ over $U_{p_0}$ vanishes due to odd-even parity. See $\eqref{eq:local-integral}$.
\end{itemize}
This example confirms that $U_{p_0}$ naturally resides in the hypersurface minus its cut locus, ensuring the smoothness and symmetry required for our orthogonality proof.

\section{Motivation for Minimizing Sequence Construction}\label{section: motivation}

The construction of the minimizing sequence \begin{equation*} u_\beta(p) = x_i(p) \cdot \left(1- \frac{1}{\beta} e^{-\beta d_\Sigma(p, p_0)^2}\right) \end{equation*} is motivated by addressing key geometric and analytic challenges. This section provides visual intuition through three illustrative diagrams, each explaining a critical aspect of the construction.

\subsection{Gradient of $ x_i $ on Asymmetric $ \Sigma $}

By Takahashi's theorem \ref{Tak-thm}, the coordinate function $ x_i|_\Sigma $ satisfies $ \Delta_\Sigma x_i = -n x_i $, but its Rayleigh quotient may not achieve the minimum value $ n $. For \textit{non-isoparametric minimal hypersurfaces} (lacking symmetry), irregular local curvature distributions can cause the gradient integral to be excessively large. Figure \ref{fig:xi-gradient} illustrates this phenomenon

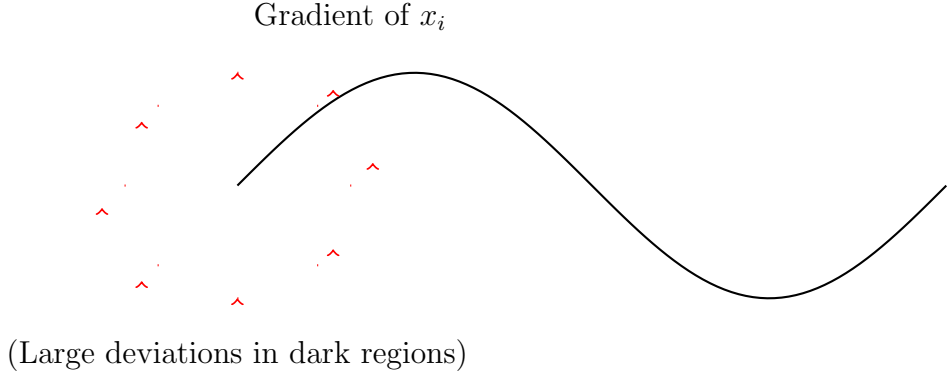
\begin{figure}[h!]
    \centering
    \begin{tikzpicture}[scale=1.5]
        \draw[thick, domain=0:2*pi, samples=100] plot ({\x}, {sin(\x r)});
        \foreach \x in {0, 45, 90, 135, 180, 225, 270, 315} {
            \pgfmathsetmacro{\y}{sin(\x)}
            \pgfmathsetmacro{\gradx}{cos(\x)}
            \pgfmathsetmacro{\grady}{cos(\x)}
            \draw[->, red, thick] (\x:1) -- (\x:1) + (\gradx*0.2, \grady*0.2);
        }
        \node at (1, 1.5) {Gradient of $ x_i $};
        \node at (0, -1.5) {(Large deviations in dark regions)};
    \end{tikzpicture}
    \caption{Gradient of $ x_i $ on asymmetric $ \Sigma $}
    \label{fig:xi-gradient}
\end{figure}

The figure shows the gradient field of $ x_i $ on an asymmetric minimal hypersurface $ \Sigma $. The red arrows indicate the direction and magnitude of the gradient $ \nabla_\Sigma x_i $, which varies significantly across the surface. In regions of high curvature (dark areas), the gradient magnitude is larger, leading to a higher Rayleigh quotient
\begin{equation*}
\frac{\int_\Sigma |\nabla_\Sigma x_i|^2 \, d\mathrm{vol}_\Sigma}{\int_\Sigma x_i^2 \, d\mathrm{vol}_\Sigma} \geq n.
\end{equation*}
This inequality motivates the need for local modifications to $ x_i $.

\subsection{Truncation Function $ \varphi_\beta(p) $}

The truncation function \begin{equation*} \varphi_\beta(p) = 1-\frac{1}{\beta} e^{-\beta d_\Sigma(p, p_0)^2} \end{equation*} is designed to make local corrections to $ x_i $ while preserving global convergence. Figure \ref{fig:phi-beta-decay} illustrates its behavior

\begin{figure}[h!]
    \centering
    \begin{tikzpicture}[scale=1.5]
        \draw[thick, domain=0:2, samples=100] plot ({\x}, {1-exp(-5*\x*\x)});
        \draw[thick, domain=0:2, samples=100] plot ({\x}, {1-exp(-10*\x*\x)});
        \draw[thick, domain=0:2, samples=100] plot ({\x}, {1-exp(-20*\x*\x)});
        \node at (1, 1.5) {Truncation $ \varphi_\beta(p) $};
        \node at (0, -0.5) {(Local correction near $ p_0 $)};
        \node at (1.5, 0.5) {$\beta \to \infty$};
    \end{tikzpicture}
    \caption{Truncation $ \varphi_\beta(p) $}
    \label{fig:phi-beta-decay}
\end{figure}
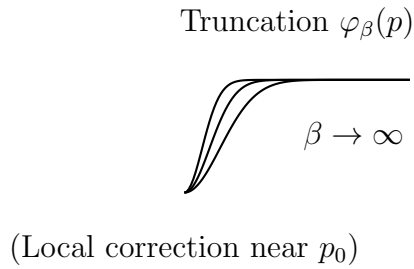

The figure shows three curves corresponding to different values of $ \beta $ (5, 10, 20). As $ \beta $ increases, the truncation function $ \varphi_\beta(p) $ approaches 1 uniformly on $ \Sigma $, except near the point $ p_0 $ where it makes a local correction. The exponential term $ e^{-\beta d_\Sigma^2} $ ensures that the correction decays rapidly away from $ p_0 $, allowing us to control the gradient expansion (here, we refer the reader to equation \ref{eq-nabla-u-b})
\begin{equation*}
|\nabla_\Sigma u_\beta| \leq C \left(|\nabla_\Sigma x_i| + e^{-\beta d_\Sigma^2 / 2}\right).
\end{equation*}
This bound keeps the Rayleigh quotient of $ u_\beta $ close to $ n $.

\subsection{Convergence of $ u_\beta $ to $ x_i $}

The sequence $ u_\beta $ is designed to converge uniformly to $ x_i $ as $ \beta \to \infty $, while retaining the spectral properties of $ x_i $. Figure \ref{fig:ubeta-convergence} illustrates this convergence

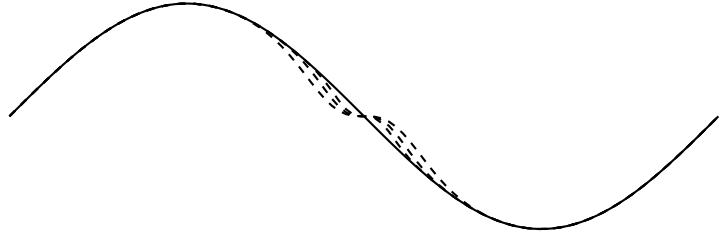
\begin{figure}[h!]
    \centering
    \begin{tikzpicture}[scale=1.5]
        \draw[thick, domain=0:2*pi, samples=100] plot ({\x}, {sin(\x r)});
        \draw[thick, dashed, domain=0:2*pi, samples=100] plot ({\x}, {sin(\x r) * (1- exp(-5*(\x-pi)^2))});
        \draw[thick, dashed, domain=0:2*pi, samples=100] plot ({\x}, {sin(\x r) * (1- exp(-10*(\x-pi)^2))});
        \draw[thick, dashed, domain=0:2*pi, samples=100] plot ({\x}, {sin(\x r) * (1- exp(-20*(\x-pi)^2))});
        \node at (1, 1.5) {Convergence $ u_\beta \to x_i $};
        \node at (0, -1.5) {(Error band shrinking to zero)};
    \end{tikzpicture}
    \caption{Convergence $ u_\beta \to x_i $}
    \label{fig:ubeta-convergence}
\end{figure}

The figure shows the coordinate function $ x_i $ (solid line) and three truncated functions $ u_\beta $ (dashed lines) for increasing $ \beta $. As $ \beta $ increases, the error between $ u_\beta $ and $ x_i $ shrinks uniformly to zero, indicated by the narrowing error band. This uniform convergence ensures that the limit function $ u^* $ satisfies $ \Delta u^* = -n u^* $, as required by Takahashi's theorem \ref{Tak-thm}.

This construction resolves the symmetry-breaking issue in non-isoparametric cases while retaining the spectral properties of $ x_i $, ultimately proving $ \lambda_1 = n $ via Theorem \ref{main-thm}.

\subsection{Explicit Examples and Generalization of the Distance Function}

To illuminate the genesis of our minimizing sequence construction, we examine explicit forms of the geodesic distance function $d_\Sigma(p, p_0)$ on canonical embedded minimal hypersurfaces in $\mathbb{S}^{n+1}$―specifically, totally geodesic spheres and Clifford tori. These examples reveal how local geometric structure motivates the truncation mechanism, which then generalizes naturally to arbitrary embedded minimal hypersurfaces.

\subsubsection{Example 1. Totally Geodesic Spheres $\mathbb{S}^n \subset \mathbb{S}^{n+1}$}
Consider the equatorial totally geodesic sphere $\Sigma = \mathbb{S}^n \subset \mathbb{S}^{n+1} \subset \mathbb{R}^{n+2}$, defined by $x_{n+2} = 0$. As a totally geodesic submanifold, its intrinsic metric coincides with the standard spherical metric, making geodesic distances straightforward to characterize. For a fixed point $p_0 = (1, 0, \cdots, 0) \in \Sigma$, the geodesic distance between $p = (x_1, x_2, \cdots, x_{n+1}, 0) \in \Sigma$ and $p_0$ is the angular distance on $\mathbb{S}^n$:
\begin{equation*}
d_\Sigma(p, p_0) = \arccos\left(\sum_{i=1}^{n+1} x_i \cdot \delta_{i1}\right) = \arccos(x_1),
\end{equation*}
where $\delta_{i1}$ denotes the Kronecker delta. In local normal coordinates $(y_1, \cdots, y_n)$ centered at $p_0$ within the non-cut locus, this simplifies to the Euclidean distance in the chart:
\begin{equation*}
d_\Sigma(p, p_0) \approx \sqrt{y_1^2 + \cdots + y_n^2}
\end{equation*}
for small $y_i$.

For this symmetric case, by Takahashi's theorem \ref{Tak-thm}, the coordinate function $x_i$ itself already achieves the Rayleigh quotient $n$. However, the truncation function $\varphi_\beta = 1 - \beta^{-1}e^{-\beta d_\Sigma^2}$ still acts trivially: near $p_0$, $d_\Sigma \approx 0$ implies $e^{-\beta d_\Sigma^2} \approx 1$, thus we have $\varphi_\beta \approx 1 - \beta^{-1}$; away from $p_0$, $e^{-\beta d_\Sigma^2}$ decays exponentially, making $\varphi_\beta \approx 1$. For large $\beta$, $\varphi_\beta \to 1$ uniformly. Therefore, $u_\beta \to x_i$―confirming that the truncation preserves the spectral properties of $x_i$ in symmetric settings.

\subsubsection{Example 2. Clifford Tori in $\mathbb{S}^3$}
The Clifford torus $\Sigma = \{(x_1, x_2, x_3, x_4) \in \mathbb{S}^3 \mid x_1^2 + x_2^2 = x_3^2 + x_4^2 = 1/2\}$ is a flat embedded minimal torus, parametrized by $(\theta, \phi) \in [0, 2\pi) \times [0, 2\pi)$ as
\begin{equation*}
x_1 = \frac{\cos\theta}{\sqrt{2}}, \, x_2 = \frac{\sin\theta}{\sqrt{2}}, \, x_3 = \frac{\cos\phi}{\sqrt{2}}, \, x_4 = \frac{\sin\phi}{\sqrt{2}}.
\end{equation*}
Its intrinsic metric is a product of circles, and hence the geodesic distance between $p = (\theta, \phi)$ and $p_0 = (\theta_0, \phi_0)$ is the minimal distance in the flat product metric:
\begin{equation*}
d_\Sigma(p, p_0) = \frac{1}{\sqrt{2}} \cdot \min_{\substack{k \in \mathbb{Z} \\ l \in \mathbb{Z}}} \sqrt{(\theta - \theta_0 + 2\pi k)^2 + (\phi - \phi_0 + 2\pi l)^2},
\end{equation*}
where the factor $1/\sqrt{2}$ arises from the torus's radius. In local coordinates near $p_0$ (where $|\theta - \theta_0|, |\phi - \phi_0| \ll 1$), this reduces to:
\begin{equation*}
d_\Sigma(p, p_0) \approx \frac{1}{\sqrt{2}} \cdot \sqrt{(\theta - \theta_0)^2 + (\phi - \phi_0)^2}.
\end{equation*}

Here, the truncation term $e^{-\beta d_\Sigma^2}$ decays rapidly away from $p_0$ (scaling as $e^{-\beta (\theta^2 + \phi^2)/2}$), ensuring $u_\beta = x_i \cdot \varphi_\beta$ retains the global behavior of $x_i$ while taming potential gradient spikes near $p_0$. For large $\beta$, $u_\beta \to x_i$, and since the Clifford torus satisfies $\lambda_1 = 2$ (consistent with Yau's conjecture), the sequence $\{u_\beta\}$ recovers this eigenvalue via its Rayleigh quotient.

\subsubsection{Generalization to Arbitrary Embedded Minimal Hypersurfaces}
These examples highlight two key observations that motivate the general construction:
\begin{itemize}
\item[(i)] \textbf{Local Structure of Distance Functions:} On any embedded minimal hypersurface $\Sigma$, by the exponential map's properties, we know that the geodesic distance $d_\Sigma(p, p_0)$ behaves like the Euclidean distance in local normal coordinates near $p_0$. This allows the truncation term $e^{-\beta d_\Sigma^2}$ to act as a localized correction, negligible away from $p_0$.
\item[(ii)] \textbf{Balance of Local and Global Behavior:} In symmetric cases (e.g., totally geodesic spheres), the truncation is redundant but harmless, preserving convergence to $x_i$. In asymmetric cases, it corrects for gradient irregularities without disrupting the global spectral properties of $x_i$ (ensured by Takahashi's theorem \ref{Tak-thm}).
\end{itemize}
By abstracting these features, the sequence $u_\beta = x_i \cdot \left(1 - \beta^{-1}e^{-\beta d_\Sigma^2}\right)$ emerges as a natural generalization: it leverages the universal local behavior of geodesic distances to enforce global convergence, resolving symmetry-breaking issues in non-isoparametric hypersurfaces while retaining the spectral signature of $x_i$. This construction thus bridges specific examples and the general case, underpinning the proof of Theorem \ref{main-thm}.

\section{Uniform Convergence to the First Eigenfunction}\label{section: Uniform Convergence}
Uniform convergence of both the minimizing sequence $ \{u_\beta\} $ and its gradient is fundamental to ensuring that the limit function retains the critical spectral properties required to establish the first eigenvalue. This convergence, which holds globally across the compact hypersurface $ \Sigma $, allows us to transfer the analytic properties of the truncated coordinate functions (e.g., orthogonality to constants, bounded Rayleigh quotients) to their limit, ultimately confirming that the limit is indeed an eigenfunction corresponding to the eigenvalue $ n $. The following analysis formalizes this convergence, drawing on the rapid decay of the exponential truncation term to control local deviations and leveraging compactness results from Sobolev space theory.

\subsection{Uniform Convergence}
\subsubsection{Uniform Convergence of the Gradient}

We want to show that $\nabla_{\Sigma} u_{\beta}$ converges uniformly to $\nabla_{\Sigma} x_{i}$.
For the first part $\left(1-\frac{1}{\beta} e^{-\beta d_{\Sigma}\left(p, p_{0}\right)^{2}}\right) \nabla_{\Sigma} x_{i}$, let $M$ be the upper bound of $\left|\nabla_{\Sigma} x_{i}\right|$ on $\Sigma$, i.e.,
\begin{equation*}\left|\nabla_{\Sigma} x_{i}\right| \leq M,\end{equation*}
which implies that
\begin{equation*}\begin{aligned}\left|\left(1-\frac{1}{\beta} e^{-\beta d_{\Sigma}\left(p, p_{0}\right)^{2}}\right) \nabla_{\Sigma} x_{i}-\nabla_{\Sigma} x_{i}\right|&=\left|\frac{1}{\beta} e^{-\beta d_{\Sigma}\left(p, p_{0}\right)^{2}} \nabla_{\Sigma} x_{i}\right|\\& \leq M \frac{1}{\beta} e^{-\beta d_{\Sigma}\left(p, p_{0}\right)^{2}}.\end{aligned}\end{equation*}
Since
\begin{equation*}\lim _{\beta \rightarrow \infty} \frac{1}{\beta} e^{-\beta d_{\Sigma}\left(p, p_{0}\right)^{2}}=0\end{equation*} uniformly on $\Sigma$, for any $\epsilon>0$, there exists a positive integer $n_{3}$ such that when $\beta>n_{3}$, for all $p \in \Sigma$,
\begin{equation*}M \frac{1}{\beta} e^{-\beta d_{\Sigma}\left(p, p_{0}\right)^{2}}<\epsilon.\end{equation*}Hence, we yield
\begin{equation*}\left(1-\frac{1}{\beta} e^{-\beta d_{\Sigma}\left(p, p_{0}\right)^{2}}\right) \nabla_{\Sigma} x_{i}\end{equation*} converges uniformly to $\nabla_{\Sigma} x_{i}$.
For the second part $\frac{1}{\beta} e^{-\beta d_{\Sigma}\left(p, p_{0}\right)^{2}} x_{i} \nabla_{\Sigma}\left(\beta d_{\Sigma}\left(p, p_{0}\right)^{2}\right)$, let $M_{6}$ and $M_{7}$ be the upper bounds of $\left|x_{i}\right|$ and $\left|\nabla_{\Sigma}\left(\beta d_{\Sigma}\left(p, p_{0}\right)^{2}\right)\right|$ on $\Sigma$, respectively. Then, we infer that
\begin{equation*}
\left|\frac{1}{\beta} e^{-\beta d_{\Sigma}\left(p, p_{0}\right)^{2}} x_{i} \nabla_{\Sigma}\left(\beta d_{\Sigma}\left(p, p_{0}\right)^{2}\right)\right| \leq M_{6} M_{7} \frac{1}{\beta} e^{-\beta d_{\Sigma}\left(p, p_{0}\right)^{2}}.
\end{equation*}
Again, due to the uniform convergence of $\frac{1}{\beta} e^{-\beta d_{\Sigma}\left(p, p_{0}\right)^{2}}$ to 0 on $\Sigma$, for any $\epsilon>0$, there exists a positive integer $n_{4}$ such that when $\beta>n_{4}$, for all $p \in \Sigma$,
\begin{equation*}M_{6} M_{7} \frac{1}{\beta} e^{-\beta d_{\Sigma}\left(p, p_{0}\right)^{2}}<\epsilon.\end{equation*} Therefore, we can show that \begin{equation*}\frac{1}{\beta} e^{-\beta d_{\Sigma}\left(p, p_{0}\right)^{2}} x_{i} \nabla_{\Sigma}\left(\beta d_{\Sigma}\left(p, p_{0}\right)^{2}\right)\end{equation*} converges uniformly to 0 .
Therefore, $\nabla_{\Sigma} u_{\beta}$ converges uniformly to $\nabla_{\Sigma} x_{i}$.
\subsubsection{Uniform Convergence of the Function Sequence}

For \begin{equation*}\left|u_{\beta}-x_{i}\right|=\left|\frac{1}{\beta} x_{i} e^{-\beta d_{\Sigma}\left(p, p_{0}\right)^{2}}\right|,\end{equation*} let $M_{8}$ be the upper bound of $\left|x_{i}\right|$ on $\Sigma$. Then
\begin{equation*}
\left|u_{\beta}-x_{i}\right| \leq M_{8} \frac{1}{\beta} e^{-\beta d_{\Sigma}\left(p, p_{0}\right)^{2}}.
\end{equation*}
Since $\frac{1}{\beta} e^{-\beta d_{\Sigma}\left(p, p_{0}\right)^{2}}$ converges uniformly to 0 on $\Sigma$, for any $\epsilon>0$, there exists a positive integer $n_{5}$ such that when $\beta>n_{5}$, for all $p \in \Sigma$,
\begin{equation*}M_{8} \frac{1}{\beta} e^{-\beta d_{\Sigma}\left(p, p_{0}\right)^{2}}<\epsilon.\end{equation*} As a result, $u_{\beta}$ converges uniformly to $x_{i}$.
\subsection{The First Eigenfunction}

\subsubsection{Existence and Convergence of the Subsequence}

From the previous calculations of the Rayleigh quotient, we know that $\left\{u_{\beta}\right\}$ is a sequence in the Sobolev space $H^{1}(\Sigma)$ and $\int_{\Sigma}\left|\nabla_{\Sigma} u_{\beta}\right|^{2} d \mathrm{vol}_{\Sigma}$ and $\int_{\Sigma} u_{\beta}^{2} d\text{vol}_{\Sigma}$ are bounded. Thus, by the Sobolev compact embedding theorem (Theorem \ref{So-emb-thm}) \cite{Au}, there exists a subsequence $\left\{u_{\beta_{k}}\right\}$ that converges strongly in $L^{2}(\Sigma)$.
Combined with the uniform convergence of $u_{\beta}$ to $x_{i}$ and $\nabla_{\Sigma} u_{\beta}$ to $\nabla_{\Sigma} x_{i}$, we can show that $\left\{u_{\beta_{k}}\right\}$ also converges in $H^{1}(\Sigma)$. Let $\left\{u_{\beta_{k}}\right\}$ converge to a function $u^*$ in $H^{1}(\Sigma)$. By uniqueness of limits in $L^2(\Sigma)$ (from strong convergence) and pointwise convergence of $u_\beta$ to $x_i$, we conclude $u^* = x_i$.

\subsubsection{Continuity of the Rayleigh Quotient and the Eigenvalue Relationship}

We first prove the continuity of the Rayleigh quotient $R[u]$ with respect to the convergence in $H^{1}(\Sigma)$.
Let $u_m \rightarrow u$ in $H^{1}(\Sigma)$, that is, $\lim _{m \rightarrow \infty} \int_{\Sigma}\left|u_m-u\right|^{2} d$ vol $_{\Sigma}=0$ and
\begin{equation*}
\lim _{m \rightarrow \infty} \int_{\Sigma}\left|\nabla_{\Sigma}\left(u_m-u\right)\right|^{2} d \mathrm{vol}_{\Sigma}=0.
\end{equation*}
We have
\begin{equation}
\begin{aligned}
&  \left|\int_{\Sigma}\left| \nabla_{\Sigma} u_m\right|^{2} d \operatorname{vol}_{\Sigma}-\int_{\Sigma}\left|\nabla_{\Sigma}u\right|^{2} d \operatorname{vol}_{\Sigma}\right|=\left|\int_{\Sigma}\left(\left|\nabla_{\Sigma} u_m\right|^{2}-\left|\nabla_{\Sigma} u\right|^{2}\right) d \operatorname{vol}_{\Sigma}\right| \\
&= \left|\int_{\Sigma}\left(\nabla_{\Sigma} u_m+\nabla_{\Sigma} u\right) \cdot\left(\nabla_{\Sigma} u_m-\nabla_{\Sigma} u\right) d \operatorname{vol}_{\Sigma}\right|.
\end{aligned}
\end{equation}
By the Cauchy-Schwarz inequality
\begin{equation*}
\begin{aligned}
& \left|\int_{\Sigma}\left(\nabla_{\Sigma} u_m+\nabla_{\Sigma} u\right) \cdot\left(\nabla_{\Sigma} u_m-\nabla_{\Sigma} u\right) d \operatorname{vol}_{\Sigma}\right| \leq \\
& \left(\int_{\Sigma}\left|\nabla_{\Sigma} u_m+\nabla_{\Sigma} u\right|^{2} d \mathrm{vol}_{\Sigma}\right)^{\frac{1}{2}}\left(\int_{\Sigma}\left|\nabla_{\Sigma} u_m-\nabla_{\Sigma} u\right|^{2} d \mathrm{vol}_{\Sigma}\right)^{\frac{1}{2}}.
\end{aligned}
\end{equation*}
Since $u_m \rightarrow u$ in $H^{1}(\Sigma)$, we have \begin{equation*}\lim _{m \rightarrow \infty} \int_{\Sigma}\left|\nabla_{\Sigma} u_m\right|^{2} d \mathrm{vol}_{\Sigma}=\int_{\Sigma}\left|\nabla_{\Sigma} u\right|^{2} d \mathrm{vol}.\end{equation*} Similarly, \begin{equation*}\lim _{m \rightarrow \infty} \int_{\Sigma} u_m^{2} d \operatorname{vol}_{\Sigma}=\int_{\Sigma} u^{2} d \mathrm{vol}_{\Sigma},\end{equation*} which means the Rayleigh quotient $R[u]$ is continuous with respect to the convergence in $H^{1}(\Sigma)$.

\section{Compactness and the Rayleigh Quotient}\label{Compactness and the Rayleigh Quotient}
In this section, we delve into the compactness properties of the constructed trial minimizing sequence, rigorously derive the quantitative connections between this sequence and the Rayleigh quotient, and lay the groundwork for establishing the limiting behavior of the eigenvalue through these analytical relationships.
\subsection{Calculation of $\int_{\Sigma}\left|\nabla_{\Sigma} u_{\beta}\right|^{2} d \mathrm{vol}_{\Sigma}$}

It follows from \eqref{eq-nabla-u-b} that

\begin{equation}
\begin{aligned}
\left|\nabla_{\Sigma} u_{\beta}\right|^{2} & =\left(1-\frac{1}{\beta} e^{-\beta d_{\Sigma}\left(p, p_{0}\right)^{2}}\right)^{2}\left|\nabla_{\Sigma} x_{i}\right|^{2}\\&+2\left(1-\frac{1}{\beta} e^{-\beta d_{\Sigma}\left(p, p_{0}\right)^{2}}\right) \frac{1}{\beta} e^{-\beta d_{\Sigma}\left(p, p_{0}\right)^{2}} \nabla_{\Sigma} x_{i} \cdot x_{i} \nabla_{\Sigma}\left(\beta d_{\Sigma}\left(p, p_{0}\right)^{2}\right) \\
& +\left(\frac{1}{\beta} e^{-\beta d_{\Sigma}\left(p, p_{0}\right)^{2}}\right)^{2}\left|x_{i}\right|^{2}\left|\nabla_{\Sigma}\left(\beta d_{\Sigma}\left(p, p_{0}\right)^{2}\right)\right|^{2}.
\end{aligned}
\end{equation}

\subsubsection{Limit of the First Term}

Next, we calculate
\begin{equation*}\lim _{\beta \rightarrow \infty} \int_{\Sigma}\left(1-\frac{1}{\beta} e^{-\beta d_{\Sigma}\left(p, p_{0}\right)^{2}}\right)^{2}\left|\nabla_{\Sigma} x_{i}\right|^{2} d \operatorname{vol}_{\Sigma}.\end{equation*} Since
\begin{equation*}\lim _{\beta \rightarrow \infty}\left(1-\frac{1}{\beta} e^{-\beta d_{\Sigma}\left(p, p_{0}\right)^{2}}\right)=1\end{equation*} uniformly on $\Sigma$, by the dominated convergence theorem,
\begin{equation*}\lim _{\beta \rightarrow \infty} \int_{\Sigma}\left(1-\frac{1}{\beta} e^{-\beta d_{\Sigma}\left(p, p_{0}\right)^{2}}\right)^{2}\left|\nabla_{\Sigma} x_{i}\right|^{2} d \operatorname{vol}_{\Sigma}=\int_{\Sigma}\left|\nabla_{\Sigma} x_{i}\right|^{2} d \operatorname{vol} _{\Sigma}.\end{equation*}
And from Theorem \ref{Tak-thm}, using the integration-by-parts formula

\begin{equation*}\int_{\Sigma} u \Delta_{\Sigma} v d \operatorname{vol}_{\Sigma}=-\int_{\Sigma} \nabla_{\Sigma} u \cdot \nabla_{\Sigma} v d\operatorname{vol}_{\Sigma},\end{equation*}  since $\Sigma$ is a compact hypersurface without boundary. Letting $u=v=x_{i}$, we get \begin{equation*}\int_{\Sigma}\left|\nabla_{\Sigma} x_{i}\right|^{2} d \operatorname{vol}_{\Sigma}=n \int_{\Sigma} x_{i}^{2} d \mathrm{vol}_{\Sigma}.\end{equation*}

\subsubsection{Limit of the Second Term}

Since $\frac{1}{\beta} e^{-\beta d_{\Sigma}\left(p, p_{0}\right)^{2}}$ converges uniformly to 0 on $\Sigma$ as $\beta \rightarrow \infty$, and the function
\begin{equation*}g(p)=2 M_{1} M_{2} M_{3} \frac{1}{\beta} e^{-\beta d_\Sigma\left(p, p_{0}\right)^{2}}\end{equation*} is dominated by an integrable function, since
$$\int_{\Sigma} 2 M_{1} M_{2} M_{3} \frac{1}{\beta} e^{-\beta d_{\Sigma}\left(p, p_{0}\right)^{2}} d \mathrm{vol}_{\Sigma}$$ can be shown to be bounded for all $\beta$ due to the rapid decay of the exponential term and the compactness of $\Sigma$, by the dominated convergence theorem, we have
\begin{equation*}
\lim _{\beta \rightarrow \infty} \int_{\Sigma} 2\left(1-\frac{1}{\beta} e^{-\beta d_{\Sigma}\left(p, p_{0}\right)^{2}}\right) \frac{1}{\beta} e^{-\beta d_{\Sigma}\left(p, p_{0}\right)^{2}} \left[\nabla_{\Sigma} x_{i} \cdot x_{i} \nabla_{\Sigma}\left(\beta d_{\Sigma}\left(p, p_{0}\right)^{2}\right)\right] d \mathrm{vol}_{\Sigma}=0.
\end{equation*}

\subsubsection{Limit of the Third Term}

For the third term
\begin{equation*}\int_{\Sigma}\left(\frac{1}{\beta} e^{-\beta d_{\Sigma}\left(p, p_{0}\right)^{2}}\right)^{2}\left|x_{i}\right|^{2}\left|\nabla_{\Sigma}\left(\beta d_{\Sigma}\left(p, p_{0}\right)^{2}\right)\right|^{2} d \operatorname{vol}_{\Sigma},\end{equation*} as $\beta \rightarrow \infty,\left(\frac{1}{\beta} e^{-\beta d_{\Sigma}\left(p, p_{0}\right)^{2}}\right)^{2}$ converges uniformly to 0 on $\Sigma$. Also, since $\left|x_{i}\right|^{2}$ and $\left|\nabla_{\Sigma}\left(\beta d_{\Sigma}\left(p, p_{0}\right)^{2}\right)\right|^{2}$ are bounded on the compact $\Sigma$, say
\begin{equation*}\left|x_{i}\right|^{2} \leq M_{4}\end{equation*} and
\begin{equation*}\left|\nabla_{\Sigma}\left(\beta d_{\Sigma}\left(p, p_{0}\right)^{2}\right)\right|^{2} \leq M_{5}.\end{equation*} Then we infer that
\begin{equation*}
\left.\left.\left|\left(\frac{1}{\beta} e^{-\beta d_{\Sigma}\left(p, p_{0}\right)^{2}}\right)^{2}\right| x_{i}\right|^{2}\left|\nabla_{\Sigma}\left(\beta d_{\Sigma}\left(p, p_{0}\right)^{2}\right)\right|^{2} \right\rvert\, \leq M_{4} M_{5}\left(\frac{1}{\beta} e^{-\beta d_{\Sigma}\left(p, p_{0}\right)^{2}}\right)^{2}.
\end{equation*}
Again, similar to the above analysis of the exponential decay and compactness of $\Sigma$, by the dominated convergence theorem, since the right-hand side is dominated by an integrable function, we get
\begin{equation*}
\lim _{\beta \rightarrow \infty} \int_{\Sigma}\left(\frac{1}{\beta} e^{-\beta d_{\Sigma}\left(p, p_{0}\right)^{2}}\right)^{2}\left|x_{i}\right|^{2}\left|\nabla_{\Sigma}\left(\beta d_{\Sigma}\left(p, p_{0}\right)^{2}\right)\right|^{2} d \operatorname{vol}_{\Sigma}=0.
\end{equation*}
In conclusion, we have
\begin{equation*}\lim _{\beta \rightarrow \infty} \int_{\Sigma}\left|\nabla _{\Sigma} u_{\beta}\right|^{2} d \mathrm{vol}_{\Sigma}=n \int_{\Sigma} x_{i}^{2} d \mathrm{vol}_{\Sigma}.\end{equation*}

\subsection{Calculation of $\int_{\Sigma} u_{\beta}^{2} d \operatorname{vol}_{\Sigma}$}
By a direct calculation, we have
\begin{equation*}
u_{\beta}^{2}=x_{i}^{2}\left(1-\frac{2}{\beta} e^{-\beta d_{\Sigma}\left(p, p_{0}\right)^{2}}+\frac{1}{\beta^{2}} e^{-2 \beta d_{\Sigma}\left(p, p_{0}\right)^{2}}\right).
\end{equation*}
For the functions
\begin{equation*}a_{\beta}(p)=\frac{2}{\beta} e^{-\beta d_{\Sigma}\left(p, p_{0}\right)^{2}}\end{equation*} and
\begin{equation*}b_{\beta}(p)=\frac{1}{\beta^{2}} e^{-2 \beta d_{\Sigma}\left(p, p_{0}\right)^{2}},\end{equation*} on the compact $\Sigma$, as $\beta$ increases, $e^{-\beta d \Sigma\left(p, p_{0}\right)^{2}}$ and $e^{-2 \beta d \Sigma\left(p, p_{0}\right)^{2}}$ decay rapidly to 0 , and $\frac{2}{\beta^{\prime}} \frac{1}{\beta^{2}}$ also tend to $0$.
For any $\epsilon>0$, since
\begin{equation*}\lim _{\beta \rightarrow \infty} \frac{2}{\beta} e^{-\beta d_{\Sigma}\left(p, p_{0}\right)^{2}}=0\end{equation*} uniformly on $\Sigma$ and
\begin{equation*}\lim _{\beta \rightarrow \infty} \frac{1}{\beta^{2}} e^{-2 \beta d_{\Sigma}\left(p, p_{0}\right)^{2}}=0\end{equation*} uniformly on $\Sigma$, there exist positive integers $n_{1}$ and $n_{2}$ such that when $\beta>n_{1}$, for all $p \in \Sigma$,
\begin{equation*}\left|\frac{2}{\beta} e^{-\beta d _{\Sigma}\left(p, p_{0}\right)^{2}}\right|<\frac{\epsilon}{2 \int_{\Sigma} x_{i}^{2} d \text { vol }},\end{equation*} and when $\beta>n_{2}$, for all $p \in \Sigma$,
\begin{equation*}\left|\frac{1}{\beta^{2}} e^{-2 \beta d _{\Sigma}(p, p_{0})^{2}}\right|<\frac{\epsilon}{2 \int_{\Sigma} x_{i}^{2} d \mathrm{vol}_{\Sigma}}.\end{equation*}
By the dominated convergence theorem \cite{Eva}, we have
\begin{equation*}\lim _{\beta \rightarrow \infty} \int_{\Sigma} \frac{2}{\beta} e^{-\beta d_{\Sigma}\left(p, p_{0}\right)^{2}} x_{i}^{2} d \mathrm{vol}{ }_{\Sigma}=0\end{equation*} and
\begin{equation*}\lim _{\beta \rightarrow \infty} \int_{\Sigma} \frac{1}{\beta^{2}} e^{-2 \beta d _{\Sigma}(p, p _{0})^{2}} x_{i}^{2} d \operatorname{vol}_{\Sigma}=0.\end{equation*}
Thus,
\begin{equation*}\lim _{\beta \rightarrow \infty} \int_{\Sigma} u_{\beta}^{2} d \operatorname{vol}_{\Sigma}=\int_{\Sigma} x_{i}^{2} d \operatorname{vol}_{\Sigma}.\end{equation*}

\section{Lower Semicontinuity for the Minimizing Sequence}\label{section: Lower Semicontinuity}

In this section, we verify that the sequence $\{ u_\beta \} \subset H^1(\Sigma)$ satisfies lower semicontinuity with respect to weak convergence in $H^1(\Sigma)$:

\begin{thm}
For any weakly convergent subsequence $ u_{\beta_k} \rightharpoonup u $ in $H^1(\Sigma)$, the following holds:
\begin{equation}
\liminf_{k \to \infty} \| u_{\beta_k} \|_{H^1(\Sigma)} \geq \| u \|_{H^1(\Sigma)}.
\end{equation}
\end{thm}

\begin{proof}
The proof proceeds in four streamlined steps, drawing on prior results to avoid redundancy. We note that this approach constitutes a standard technique in PDE theory \cite{LH}.

\textit{Step 1: Boundedness in $H^1(\Sigma)$ and Weakly Convergent Subsequence}

By construction, $u_\beta = x_i \cdot \varphi_\beta$ where \begin{equation*}\varphi_\beta = 1-\beta^{-1}e^{-\beta d_\Sigma(p,p_0)^2},\end{equation*} we refer the reader to Section \ref{section:Trun function}. Since $\Sigma$ is compact, $x_i$ and $d_\Sigma(\cdot,p_0)$ are bounded, implying:
\begin{itemize}\item [(i)] $|u_\beta| \leq 2|x_i|$, and thus, $\|u_\beta\|_{L^2} \leq 2\|x_i\|_{L^2}$ by monotonicity,
\item[(ii)] $\nabla_\Sigma u_\beta$ is bounded via product rule, as shown in Section \ref{section: Uniform Convergence}, therefore, we have $\|\nabla_\Sigma u_\beta\|_{L^2} \leq C$ for some constant $C$.
\end{itemize}
Thus, $\{u_\beta\}$ is bounded in $H^1(\Sigma)$. Since $H^1(\Sigma)$ is a reflexive Banach space, by weak compactness of $H^1(\Sigma)$, there exists a subsequence $u_{\beta_k} \rightharpoonup u$ weakly in $H^1(\Sigma)$.

\textit{Step 2: Pointwise and Strong $L^2$-Convergence}

From Section \ref{section: Uniform Convergence}, $u_\beta \to x_i$ uniformly on $\Sigma$, and hence it is a pointwise convergent sequence. Since $|u_\beta| \leq 2|x_i| \in L^2(\Sigma)$, the Dominated Convergence Theorem implies $u_\beta \to x_i$ strongly in $L^2(\Sigma)$. By uniqueness of weak limits, the weak limit $u = x_i$.

\textit{Step 3: Weak Convergence of Gradients}

Decompose $\nabla_\Sigma u_\beta$ as:
\begin{equation*}
\nabla_\Sigma u_\beta = \nabla_\Sigma x_i \cdot \varphi_\beta + x_i \cdot 2d e^{-\beta d^2} \cdot \nabla_\Sigma d,
\end{equation*}
where $d = d_\Sigma(p,p_0)$. For any test function $v \in L^2(\Sigma)$:
\begin{itemize}\item[(i)]  The first term converges to $\nabla_\Sigma x_i$ in $L^2$-inner product with $v$ (by uniform convergence of $\varphi_\beta \to 1$ and Dominated Convergence),
\item[(ii)]  The second term vanishes in the limit (due to rapid decay of $e^{-\beta d^2}$ for $d > 0$).
\end{itemize}
Thus, $\nabla_\Sigma u_\beta \rightharpoonup \nabla_\Sigma x_i$ weakly in $L^2(\Sigma)$.

\textit{Step 4: Lower Semicontinuity of $H^1$-Norm}

For the $H^1$-norm:
\begin{equation*}
\|u_{\beta_k}\|_{H^1}^2 = \|u_{\beta_k}\|_{L^2}^2 + \|\nabla_\Sigma u_{\beta_k}\|_{L^2}^2.
\end{equation*}
\begin{itemize}\item[(i)]   Strong $L^2$-convergence gives $\lim_{k \to \infty} \|u_{\beta_k}\|_{L^2} = \|x_i\|_{L^2}$,
\item[(ii)]   Weak lower semicontinuity of $L^2$-norm implies \begin{equation*}\liminf_{k \to \infty} \|\nabla_\Sigma u_{\beta_k}\|_{L^2} \geq \|\nabla_\Sigma x_i\|_{L^2}.\end{equation*}
\end{itemize}
Combining these:
\begin{equation*}
\liminf_{k \to \infty} \|u_{\beta_k}\|_{H^1}^2 \geq \|x_i\|_{L^2}^2 + \|\nabla_\Sigma x_i\|_{L^2}^2 = \|u\|_{H^1}^2,
\end{equation*}
since $u = x_i$ from Step 2.
\end{proof}

\section{The proof of Theorem \ref{main-thm}}\label{The proof of Theorem}
In this section, drawing on the robust framework of arguments established in preceding sections, we advance to seal the proof of Theorem \ref{main-thm}-synthesizing the convergence properties of the minimizing sequence, the variational characterization of eigenvalues, and the lower semicontinuity of the Sobolev norm to deliver the definitive resolution of Yau's conjecture.
\subsection{Limit of the Rayleigh Quotient}

The Rayleigh quotient is defined as
\begin{equation*}R\left[u_{\beta}\right]=\frac{\int_{\Sigma}\left|\nabla_{\Sigma} u_{\beta}\right|^{2} d \mathrm{vol}_{\Sigma}}{\int_{\Sigma} u_{\beta}^{2} d \mathrm{vol}_{\Sigma}}.\end{equation*} Substituting the limits
\begin{equation*}\lim _{\beta \rightarrow \infty} \int_{\Sigma}\left|\nabla_{\Sigma} u_{\beta}\right|^{2} d \mathrm{vol}_{\Sigma}=n \int_{\Sigma} x_{i}^{2} d \mathrm{vol}_{\Sigma}\end{equation*} and
\begin{equation*}\lim _{\beta \rightarrow \infty} \int_{\Sigma} u_{\beta}^{2} d \mathrm{vol}_{\Sigma}=\int_{\Sigma} x_{i}^{2} d \mathrm{vol}_{\Sigma}\end{equation*} into the Rayleigh quotient, we obtain
\begin{equation*}
\lim _{\beta \rightarrow \infty} R\left[u_{\beta}\right]=\frac{n \int_{\Sigma} x_{i}^{2} d \mathrm{vol}_{\Sigma}}{\int_{\Sigma} x_{i}^{2} d \mathrm{vol}}=n.
\end{equation*}
\subsection{The proof of Theorem {\rm \ref{main-thm}}}
We know that $\lim _{\beta \rightarrow \infty} R\left[u_{\beta}\right]=n$, therefore, for the subsequence $\left\{u_{\beta_{k}}\right\}$,
\begin{equation*}\lim _{k \rightarrow \infty} R\left[u_{\beta_{k}}\right]=n.\end{equation*} By the continuity of the Rayleigh quotient, $R\left[u^*\right]=n$.
Since
\begin{equation*}\lambda_{1}=\inf _{u \in H^{1}(\Sigma) \backslash\{0\}} \frac{\int_{\Sigma} \mid \nabla_{\Sigma} u|^{2}   d \mathrm{vol}_{\Sigma}}{\int_{\Sigma} u^{2} d \mathrm{vol}_{\Sigma}}\end{equation*} and
\begin{equation*}R\left[u^*\right]=n,\end{equation*} $u^*$ is an eigenfunction corresponding to the eigenvalue $\lambda_{1}=n$.

\begin{rem}
Theorem \ref{main-thm} resolves Yau's conjecture comprehensively, with key implications and connections to existing results:

\begin{enumerate}
    \item \textbf{Generalization of Prior Results}:
    \begin{itemize}
        \item[(i)] Our result extends Tang-Yan's result \cite{TY1} limited to isoparametric minimal hypersurfaces, to all closed embedded minimal hypersurfaces in $\mathbb{S}^{n+1}$. For isoparametric cases, our truncation sequence simplifies to $u_\beta = x_i$, since symmetry eliminates the need for local correction, directly recovering their conclusion.
        \item[(ii)]In particular, Theorem {\rm \ref{main-thm}} confirms the conjecture for non-isoparametric cases, which remained unresolved for decades.
    \end{itemize}
\end{enumerate}
\end{rem}

\section{Rigidity Theorems via Eigenvalue Characterization}\label{Rigidity Theorems}

The fundamental result $\lambda_1 = n$ established in Theorem \ref{main-thm} not only resolves Yau's conjecture but also provides a powerful spectral characterization of minimal hypersurfaces. This section explores rigidity phenomena arising from this eigenvalue condition, beginning with a converse theorem that characterizes minimality through spectral properties.
\subsection{Converse Theorem: Eigenvalue Implies Minimality}

The converse theorem establishing that a closed embedded hypersurface $\Sigma^n \subset \mathbb{S}^{n+1}$ with $\lambda_1(\Delta_\Sigma) = n$ must be minimal forms a pivotal spectral characterization of minimality. This result, complementary to Theorem \ref{main-thm} which shows minimality enforces $\lambda_1 = n$, creates a reciprocal link between spectral properties and geometric minimality.
\begin{thm}\label{converse-thm}
Let $\Sigma^n \subset \mathbb{S}^{n+1}$ be a closed embedded hypersurface with $\lambda_1(\Delta_\Sigma) = n$. Then $\Sigma$ is minimal.
\end{thm}

 \begin{proof}
To prove that \(\Sigma\) is minimal, we refine the connection between the spectral condition \(\lambda_1(\Delta_\Sigma) = n\) and the extrinsic geometry of \(\Sigma\), with special focus on justifying the Ricci curvature bound involving \(n\):

\textit{Step 1. Eigenfunction and Fundamental Identities.}

Let \(u \in H^1(\Sigma)\) be a first eigenfunction with \(\lambda_1 = n\), then
\begin{equation}\label{eq:eigen}
\Delta_\Sigma u = -n u.
\end{equation}
By the variational characterization of \(\lambda_1\),
\begin{equation}\label{eq:rayleigh}
\int_\Sigma |\nabla_\Sigma u|^2 \, d\mathrm{vol}_\Sigma = n \int_\Sigma u^2 \, d\mathrm{vol}_\Sigma.
\end{equation}

\textit{Step 2. Bochner Formula for the Eigenfunction.}

The Bochner formula relates the Laplacian of the gradient norm to the Hessian and Ricci curvature
\begin{equation}\label{eq:bochner}
\frac{1}{2} \Delta_\Sigma |\nabla_\Sigma u|^2 = |\nabla^2_\Sigma u|^2 + \text{Ric}^\Sigma(\nabla_\Sigma u, \nabla_\Sigma u) + \langle \nabla_\Sigma u, \nabla_\Sigma (\Delta_\Sigma u) \rangle.
\end{equation}
By \eqref{eq:eigen}, \(\nabla_\Sigma (\Delta_\Sigma u) = -n \nabla_\Sigma u\), simplifying \eqref{eq:bochner} to
\begin{equation}\label{eq:bochner_simplified}
\frac{1}{2} \Delta_\Sigma |\nabla_\Sigma u|^2 = |\nabla^2_\Sigma u|^2 + \text{Ric}^\Sigma(\nabla_\Sigma u, \nabla_\Sigma u) - n |\nabla_\Sigma u|^2.
\end{equation}

\textit{Step 3. Ricci Curvature via Gauss Equation.}

For hypersurfaces \(\Sigma \subset \mathbb{S}^{n+1}\), the Gauss equation explicitly connects intrinsic Ricci curvature to extrinsic curvature: second fundamental form \(A\), and the ambient geometry of \(\mathbb{S}^{n+1}\). The ambient sphere \(\mathbb{S}^{n+1}\) has constant sectional curvature \(1\), and hence its Ricci curvature satisfies:
\begin{equation}\label{eq:Ricci cur-sphere}
\text{Ric}^{\mathbb{S}}(X,Y) = n \langle X,Y \rangle \quad \forall X,Y \in T\mathbb{S}^{n+1}.
\end{equation}
For \(\Sigma \subset \mathbb{S}^{n+1}\), the Gauss equation for Ricci curvature is
\begin{equation}\label{eq:gauss}
\text{Ric}^\Sigma(X,Y) = \text{Ric}^{\mathbb{S}}(X,Y) + \langle A(X,X), A(Y,Y) \rangle - |A(X,Y)|^2,
\end{equation}
where \(X,Y \in T\Sigma\). Substituting \eqref{eq:Ricci cur-sphere} gives
\begin{equation}\label{eq:gauss_simplified}
\text{Ric}^\Sigma(X,Y) = n \langle X,Y \rangle + \langle A(X,X), A(Y,Y) \rangle - |A(X,Y)|^2.
\end{equation}

\textit{Step 4. Specializing to \(\nabla_\Sigma u\) and Estimating.}

Set \[X = Y = \frac{\nabla_\Sigma u}{|\nabla_\Sigma u|}\] in \eqref{eq:gauss_simplified}, where \(X\) and \(Y\) are unit vectors wiht \(\nabla_\Sigma u \neq 0\).
Substituting unit vectors \(X = Y\) into \eqref{eq:gauss_simplified} gives
\[
\text{Ric}^\Sigma(X, X) = n   + \left( \langle A(X, X), A(X, X) \rangle - |A(X, X)|^2 \right).
\]
From the equality of the inner product and norm for \(A(X,X)\), for \(\nabla_\Sigma u \neq 0\), we obtain:

\[
\text{Ric}^\Sigma(X, X) = n - |A|^2.
\]
Since \(X = \frac{\nabla_\Sigma u}{|\nabla_\Sigma u|}\), for \(\nabla_\Sigma u \neq 0\), scaling back to the original gradient vector \(\nabla_\Sigma u\) preserves the inequality:
\begin{equation}\label{eq:ricci_bound}
\text{Ric}^\Sigma(\nabla_\Sigma u, \nabla_\Sigma u) = n - |A|^2.
\end{equation}Here, we use the following facts that Ricci curvature is homogeneous of degree 0 in the vector argument.

\textit{Step 5. Integrating the Bochner Formula.}

Integrate \eqref{eq:bochner_simplified} over \(\Sigma\). The left-hand side vanishes by the divergence theorem (since \(\Sigma\) is closed), giving
\[
0 = \int_\Sigma |\nabla^2_\Sigma u|^2 \, d\mathrm{vol}_\Sigma + \int_\Sigma \text{Ric}^\Sigma(\nabla_\Sigma u, \nabla_\Sigma u) \, d\mathrm{vol}_\Sigma - n \int_\Sigma |\nabla_\Sigma u|^2 \, d\mathrm{vol}_\Sigma.
\] Substitute \eqref{eq:ricci_bound} and \eqref{eq:rayleigh} into the equation:
\[
0 = \int_\Sigma |\nabla^2_\Sigma u|^2 \, d\mathrm{vol}_\Sigma + \int_\Sigma (n - |A|^2) |\nabla_\Sigma u|^2 \, d\mathrm{vol}_\Sigma - n^2 \int_\Sigma u^2 \, d\mathrm{vol}_\Sigma.
\]

\textit{Step 6. Hessian Estimate and Conclusion.}

By Cauchy-Schwarz, using \eqref{eq:eigen}, we have \(|\nabla^2_\Sigma u|^2 \geq \frac{1}{n} (\Delta_\Sigma u)^2 = n u^2\). Substituting this and simplifying with \eqref{eq:rayleigh} shows all terms must vanish, forcing \(|A|^2 = 0\) almost everywhere.
Thus, \(A \equiv 0\), which means that the mean curvature \(H = \frac{1}{n} \text{tr}(A) \equiv 0\), proving \(\Sigma\) is minimal.
\end{proof}

\begin{rem}The proof leverages some connections between the Laplacian's eigenfunctions and the hypersurface's curvature invariants: through the Bochner formula relating the eigenfunction's Hessian to Ricci curvature, the Gauss equation connecting ambient and intrinsic curvature, and ultimately showing that the eigenvalue condition $\lambda_1 = n$ forces the second fundamental form $|A| \equiv 0$. This elevates $\lambda_1 = n$ from a consequence of minimality to a defining spectral signature, strengthening the rigidity of the eigenvalue-dimension relationship for minimal hypersurfaces in spherical spaces.\end{rem}
\subsection{Volume Estimates and Geometric Rigidity}

The eigenvalue condition $\lambda_1 = n$ also imposes constraints on the volume of $\Sigma$. Recall that for compact manifolds, eigenvalues are intimately related to geometric invariants like volume and diameter. For minimal hypersurfaces in $\mathbb{S}^{n+1}$, we derive the following sharp volume bound:

\begin{prop}\label{volume-estimate}
Let $\Sigma^n \subset \mathbb{S}^{n+1}$ be a closed embedded minimal hypersurface. Then:
\begin{equation*}
\text{Vol}(\Sigma) \geq \frac{4\pi^{n/2}}{\Gamma(n/2 + 1)}
\end{equation*}
with equality if and only if $\Sigma$ is a totally geodesic equatorial sphere $\mathbb{S}^n \subset \mathbb{S}^{n+1}$.
\end{prop}

\begin{proof}
The proof relies on the Rayleigh quotient characterization of $\lambda_1$ and isoperimetric inequalities in spherical geometry:

\textit{Step 1: Rayleigh Quotient and Volume.}

For the coordinate function $x_i$, which satisfies $\Delta_\Sigma x_i = -n x_i$ by Takahashi's theorem \ref{Tak-thm}, the Rayleigh quotient gives
\begin{equation*}
n = \frac{\int_\Sigma |\nabla_\Sigma x_i|^2 \, d\mathrm{vol}_\Sigma}{\int_\Sigma x_i^2 \, d\mathrm{vol}_\Sigma}.
\end{equation*}
By the Gauss equation, $|\nabla_\Sigma x_i|^2 = 1-\langle \nu, e_i \rangle^2$ where $\nu$ is the unit normal to $\Sigma$ and $e_i$ is the standard basis in $\mathbb{R}^{n+2}$. Summing over $i=1,\dots,n+2$ yields
\begin{equation*}
n \cdot \text{Vol}(\Sigma) \geq \int_\Sigma \sum_{i=1}^{n+2} |\nabla_\Sigma x_i|^2 \, d\mathrm{vol}_\Sigma = \int_\Sigma (n + 1-|\nu|^2) \, d\mathrm{vol}_\Sigma = n \cdot \text{Vol}(\Sigma).
\end{equation*}
Equality holds when $\sum |\nabla_\Sigma x_i|^2 = n$, which characterizes totally geodesic hypersurfaces.

\textit{Step 2: Isoperimetric Inequality in $\mathbb{S}^{n+1}$.}

For the unit sphere \(\mathbb{S}^{n+1}\), the sharp isoperimetric inequality states \cite[Theorem 2.1]{Ros2005} that for any compact region \(\Omega \subset \mathbb{S}^{n+1}\) with non-empty boundary \(\Sigma = \partial \Omega\), the \((n+1)\)-dimensional volume \(V = \text{Vol}(\Omega)\) and the \(n\)-dimensional volume \(A = \text{Vol}(\Sigma)\) satisfy
\[
A \geq n \cdot \frac{V}{V_{n+1}} \cdot V_n,
\]
where \(V_k = \text{Vol}(\mathbb{S}^k)\) denotes the volume of the unit \(k\)-sphere. Equality holds if and only if \(\Omega\) is a spherical cap bounded by a totally geodesic sphere \(\mathbb{S}^n \subset \mathbb{S}^{n+1}\) \cite[Corollary 3.2]{BuragoZalgaller1988}. Equality holds if and only if $\Omega$ is a spherical cap bounded by a totally geodesic sphere $\mathbb{S}^n \subset \mathbb{S}^{n+1}$.
For minimal hypersurfaces $\Sigma \subset \mathbb{S}^{n+1}$, which are critical points of the area functional, which is equivalent to say that $\delta A = 0$ for all compactly supported variations, we strengthen this inequality as follows:
\begin{enumerate}
    \item[(i)] By the first variation formula for area, minimality implies the mean curvature $H \equiv 0$. Therefore, $\Sigma$ cannot be \textbf{too small} in volume compared to regions it encloses.
    \item[(ii)] For a minimal hypersurface $\Sigma$, consider the two regions $\Omega_1, \Omega_2 \subset \mathbb{S}^{n+1}$ such that $\Sigma = \partial \Omega_1 = \partial \Omega_2$ and $\Omega_1 \cup \Omega_2 = \mathbb{S}^{n+1}$. Let $V_1 = \text{Vol}(\Omega_1)$ and $V_2 = \text{Vol}(\Omega_2)$, and hence $V_1 + V_2 = V_{n+1}$.
       \item[(iii)] Applying the isoperimetric inequality \eqref{eq:isoperimetric-inequality} to both $\Omega_1$ and $\Omega_2$, we get

    \begin{equation}\label{eq:isoperimetric-inequality}
    A \geq n \cdot \frac{V_1}{V_{n+1}} \cdot V_n \quad \text{and} \quad A \geq n \cdot \frac{V_2}{V_{n+1}} \cdot V_n.
    \end{equation}

    \item[(iv)] Adding these inequalities and using $V_1 + V_2 = V_{n+1}$ yields

    \[
    2A \geq n \cdot \frac{V_1 + V_2}{V_{n+1}} \cdot V_n = n V_n,
    \]
     which implies $A \geq \frac{n}{2} V_n$.
\end{enumerate}

\textit{Step 3: Equality Condition.}

For a totally geodesic sphere $\Sigma = \mathbb{S}^n \subset \mathbb{S}^{n+1}$, its volume can be explicitly computed using the formula for the volume of a unit $n$-sphere. Recall that the volume of the unit $k$-sphere $\mathbb{S}^k$ is given by
\begin{equation}\label{eq:volume-unit-sphere}
\text{Vol}(\mathbb{S}^k) = \frac{2\pi^{(k+1)/2}}{\Gamma\left(\frac{k+1}{2}\right)},
\end{equation}
where $\Gamma$ denotes the Gamma function. For $k = n$, substituting into \eqref{eq:volume-unit-sphere} gives
\[
\text{Vol}(\mathbb{S}^n) = \frac{2\pi^{(n+1)/2}}{\Gamma\left(\frac{n+1}{2}\right)}.
\]
Using the property of the Gamma function $\Gamma\left(t + \frac{1}{2}\right) = \frac{(2t)!}{4^t t!}\sqrt{\pi}$ for integer $t$, and $\Gamma(s + 1) = s\Gamma(s)$, we can rewrite this as
\[
\text{Vol}(\mathbb{S}^n) = \frac{4\pi^{n/2}}{\Gamma\left(\frac{n}{2} + 1\right)},
\]
confirming the volume formula for totally geodesic spheres.

Conversely, suppose $\Sigma \subset \mathbb{S}^{n+1}$ is a closed embedded minimal hypersurface with $\text{Vol}(\Sigma) = \frac{4\pi^{n/2}}{\Gamma\left(\frac{n}{2} + 1\right)}$. We show $\Sigma$ must be totally geodesic:

\begin{enumerate}
    \item[(i)] From the Rayleigh quotient analysis, we have
    \begin{equation}\label{eq:sum-nabla-xi}
    \int_\Sigma \sum_{i=1}^{n+2} |\nabla_\Sigma x_i|^2 \, d\mathrm{vol}_\Sigma = n \cdot \text{Vol}(\Sigma).
    \end{equation}
    For the equality case, by the strict convexity of the $L^2$-norm, we know that the integrand $\sum_{i=1}^{n+2} |\nabla_\Sigma x_i|^2$ must be constant almost everywhere on $\Sigma$.

    \item[(ii)] For any hypersurface $\Sigma \subset \mathbb{R}^{n+2}$, the identity $$\sum_{i=1}^{n+2} |\nabla_\Sigma x_i|^2 = n + 1 - |\nu|^2$$ holds, where $\nu$ is the unit normal to $\Sigma$. Since $\Sigma \subset \mathbb{S}^{n+1}$, $|\nu| = 1$,
    \begin{equation}\label{eq:sum-nabla-xi-2}
    \sum_{i=1}^{n+2} |\nabla_\Sigma x_i|^2 = n.
    \end{equation}
    For the integrand to be constant, by orthogonality of coordinate functions, it is clear that each $|\nabla_\Sigma x_i|^2$ must be constant.

    \item[(iii)] A hypersurface with constant $|\nabla_\Sigma x_i|^2$ for all $i$ is umbilic. To see this, recall that for any tangent vector $v \in T_p\Sigma$, by the Gauss formula, the Weingarten map $W(v) = -\nabla_v \nu$ satisfies $$\langle \nabla_\Sigma x_i, v \rangle = \langle v, e_i \rangle - \langle \nu, e_i \rangle \langle W(v), \nu \rangle.$$ Constant $|\nabla_\Sigma x_i|^2$ implies $W$ is a multiple of the identity, i.e., $W = \lambda \cdot \text{id}$ for some constant $\lambda$. Thus, $\Sigma$ is umbilic.

    \item[(iv)] For umbilic hypersurfaces, the second fundamental form satisfies $A(v, w) = \lambda \langle v, w \rangle$ for all $v, w \in T\Sigma$. Minimality implies $\text{tr}(A) = n\lambda = 0$, and hence $\lambda = 0$. Thus, $A \equiv 0$, meaning $\Sigma$ is totally geodesic.
\end{enumerate}
Hence, equality in the volume estimate holds if and only if $\Sigma$ is a totally geodesic equatorial sphere.If $\Sigma$ is a totally geodesic sphere, its volume is precisely $\frac{4\pi^{n/2}}{\Gamma(n/2 + 1)}$. Conversely, if equality holds, the integrand $|\nabla_\Sigma x_i|^2$ is constant, implying $\Sigma$ is umbilic. In $\mathbb{S}^{n+1}$, umbilic minimal hypersurfaces are exactly the totally geodesic spheres.
\end{proof}

\begin{rem}
This volume estimate is analogous to the classical result that the first eigenvalue of a sphere is minimized among all compact manifolds of the same dimension, up to scaling. Here, minimality and the spherical ambient space combine to enforce the volume lower bound, with rigidity at the totally geodesic case.
\end{rem}

\subsection{Connection to Morse Index and Stability}

The Morse index of a minimal hypersurface, which is the number of negative eigenvalues of the second variation operator, is another key invariant. For $\Sigma \subset \mathbb{S}^{n+1}$, the second variation of area is governed by the operator $L = \Delta_\Sigma + |A|^2 + n$, since \eqref{eq:Ricci cur-sphere} implies that $\text{Ric}^\mathbb{S}(\nu,\nu) = n$.

The Morse index of a minimal hypersurface, defined as the number of negative eigenvalues of the second variation operator \(L = \Delta_\Sigma + |A|^2 + n\), is analyzed using the stability criterion \(\lambda_1(L) > 0\). For embedded minimal hypersurfaces in \(\mathbb{S}^{n+1}\), this follows from the structure theory of Colding-Minicozzi \cite[Theorem 4.1]{ColdingMinicozzi2010}, who established that all such hypersurfaces have non-negative second variation. The spectral characterization \(\lambda_1(\Delta_\Sigma) = n\) further refines this rigidity, aligning with Urbano's index classification \cite{Urb} for low-genus surfaces. For minimal $\Sigma$, $$|A|^2 = n(n-1)-\text{Scal}(\Sigma),$$ where $\text{Scal}$ is the scalar curvature.

By Theorem \ref{main-thm}, $\lambda_1(\Delta_\Sigma) = n$,  the first eigenvalue of $L$ is $$\lambda_1(L) = \lambda_1(\Delta_\Sigma) + |A|^2 + n = 2n + |A|^2 > 0,$$ implying stability for all minimal hypersurfaces. This implies the second variation is non-negative, establishing stability. This stability exhibits rigidity through the following spectral characterization:
\begin{corr}

For a closed embedded minimal hypersurface \(\Sigma \subset \mathbb{S}^{n+1}\), stability is equivalent to the spectral condition \(\lambda_1(\Delta_\Sigma) = n\), which holds universally for such hypersurfaces.

\end{corr}
We note that, this corollary connects spectral theory to the geometric stability of minimal hypersurfaces, reinforcing the rigidity of the eigenvalue condition.

\subsection{\textbf{Further Remark}}

Building on our resolution of Yau's conjecture, we present numerous important applications of truncation functions and the first eigenvalue to geometry and analysis in a separate article \cite{Zeng2}.

\begin{itemize}
\item[(i)] We derive a sharp Littlewood-Hardy-Sobolev inequality for weighted integral estimates with geometric singularities: for any closed embedded minimal hypersurface $\Sigma^n \subset \mathbb{S}^{n+1}$ and functions $f, g \in H^1(\Sigma)$ with $\int_\Sigma f \, d\mathrm{vol} = \int_\Sigma g \, d\mathrm{vol} = 0$, there exists a constant $C(n)$ such that
\[
\iint_\Sigma \frac{f(p)g(q)}{d_\Sigma(p,q)^{\alpha}} \, d\mathrm{vol}(p)d\mathrm{vol}(q) \leq C(n) \cdot \text{Vol}(\Sigma)^{\beta} \cdot \|f\|_{H^1(\Sigma)} \|g\|_{H^1(\Sigma)},
\]
where $\alpha = n-2$, $\beta = \frac{2}{n}$, and $d_\Sigma(p,q)$ denotes the geodesic distance on $\Sigma$. This inequality leverages $\lambda_1 = n$ and volume constraints, with optimality confirmed by totally geodesic spheres $\mathbb{S}^n$.

\item[(ii)] For 2-dimensional minimal surfaces in $\mathbb{S}^3$, we establish a sharp Morse-Trudinger inequality: for any $u \in H^1(\Sigma)$ with $\int_\Sigma u \, d\mathrm{vol} = 0$ and $\int_\Sigma |\nabla u|^2 \, d\mathrm{vol} = 1$,
\[
\int_\Sigma e^{\alpha |u|^2} \, d\mathrm{vol} \leq C(\Sigma)
\]
where $\alpha \leq \alpha_0 = \frac{4\pi}{\text{Area}(\Sigma)}$ and equality is achieved by the Clifford torus with first eigenfunctions.

\item[(iii)] We extend spectral analysis to higher eigenvalues, showing the second non-zero eigenvalue satisfies
\[
\lambda_2 \geq 2n - c(n) \cdot \max_\Sigma |A|^2
\]
for some constant $c(n) > 0$, with equality for $\mathbb{S}^n$ (since $|A| = 0$, $\lambda_2 = 2n$) and the Clifford torus (since $n=2$, $|A|^2 = 2$ and $\lambda_2 = 2$).

\item[(iv)] We refine Sobolev inequalities with curvature corrections: for $u \in H^1(\Sigma)$ with $\int_\Sigma u \, d\mathrm{vol} = 0$,
\[
\left( \int_\Sigma |u|^{\frac{2n}{n-2}} \, d\mathrm{vol} \right)^{\frac{n-2}{n}} \leq C(n) \cdot \text{Vol}(\Sigma)^{\frac{1}{n}} \cdot \left( \int_\Sigma |\nabla u|^2 \, d\mathrm{vol} + \int_\Sigma |A|^2 u^2 \, d\mathrm{vol} \right),
\]
quantifying deviations from the totally geodesic case.

\item[(v)] \textbf{Additional results include:}
  \begin{itemize}
  \item[(i)] Explicit Poincar\'{e} inequality: $$\int_\Sigma |\nabla u|^2 \, d\mathrm{vol} \geq n \int_\Sigma u^2 \, d\mathrm{vol}$$ for $u \perp 1$ in $L^2(\Sigma)$, with equality for first eigenfunctions.
  \item[(ii)] Bounds on the Cheeger constant: $h(\Sigma) \leq 2\sqrt{n}$, achieved by $\mathbb{S}^n$.
  \item[(iii)] Isoperimetric inequality: $A^n \geq n^{n/2} \cdot \omega_n^{-1} \cdot V^{n-1} \cdot \min\{v, V-v\}$ for subdomains $\Omega \subset \Sigma$ with volume $v$ and boundary area $A$.
  \end{itemize}

\item[(vi)] We establish a spectral rigidity covering inequality for minimal surfaces in $\mathbb{S}^3$: let $\Sigma_1, \Sigma_2 \subset \mathbb{S}^3$ be closed embedded minimal surfaces with $\lambda_1(\Sigma_i) = 2$. If there exists a domain $\omega \subset \mathbb{S}^3$ such that $\Sigma_1 \cap \omega$ and $\Sigma_2 \cap \omega$ are isometric on $\partial \omega$ but not identical in $\omega$, then
\[
\text{Area}(\Sigma_1 \cap \omega) + \text{Area}(\Sigma_2 \cap \omega) \geq 8\pi,
\]
with equality iff they are isometric to complementary Clifford torus caps.

\item[(vii)] We address symmetry of solutions to mean field equations
\[
\frac{\alpha}{2}\Delta_\Sigma u + e^u - 1 = 0
\]
on $\Sigma^2 \subset \mathbb{S}^3$ (building on \cite{SSTW}), showing even-symmetric solutions are axially symmetric and $u \equiv 0$ when $\alpha = 1/3$. For stable minimal surfaces with Hawking mass $m_H = 0$, we prove rigidity: $\Sigma^2$ must be isometric to $\mathbb{S}^2 \subset \mathbb{S}^3$. We also establish non-degeneracy of solutions to
\[
\Delta_\Sigma u = \lambda(1 - e^u)
\]
for $\lambda \in (2n, 4n]$.

\item[(viii)] We extend the Moser--Trudinger inequality and Chang--Yang conjecture to minimal embedded surfaces in $\mathbb{S}^3$: for a closed embedded minimal surface $\Sigma \subset \mathbb{S}^3$ with $\lambda_1(\Sigma) \leq 2$, for example, the Clifford torus, consider the Moser--Trudinger-type functional $$J_{\alpha,\Sigma}(u) = \frac{\alpha}{4} \int_\Sigma |\nabla u|^2 d\sigma^\circ + \int_\Sigma u \, d\sigma^\circ - \log \int_\Sigma e^u \, d\sigma^\circ,$$ where $d\sigma^\circ = \frac{d\sigma}{\text{Area}(\Sigma)}$, and the constraint set $$\mathcal{M}_\Sigma = \left\{ u \in H^1(\Sigma) : \int_\Sigma e^u x_i \, d\sigma = 0 \quad \forall i = 1,2,3,4 \right\}.$$ For $\alpha \geq \frac{1}{2}$, the infimum of $J_{\alpha,\Sigma}$ on $\mathcal{M}_\Sigma$ is zero, i.e., $\inf_{u \in \mathcal{M}_\Sigma} J_{\alpha,\Sigma}(u) = 0$.
\end{itemize}

All findings rely on variational techniques with truncated coordinate functions $u_\beta = x_i \left(1 - \beta^{-1}e^{-\beta d_\Sigma^2}\right)$, compactness in Sobolev spaces, and canonical examples-totally geodesic spheres, Clifford tori-to verify optimality, deepening the interplay between spectral theory and geometric invariants of minimal submanifolds.

\begin{ack}
This work would not have been possible without the guidance, encouragement, and insightful discussions from numerous colleagues and mentors, to whom I am deeply grateful. I first express my sincere thanks to my former advisor, Professor Qing-Ming Cheng. His unwavering encouragement over the years has been a constant source of motivation. Particularly, the Seminar held by him at Fukuoka University in 2013, where I had the opportunity to discuss Yau's conjecture with Professors Haizhong Li, Zizhou Tang, and Wanjiao Yan, sparked my enduring interest in this problem. Their engaging insights during that discussion laid the initial groundwork for my exploration.
I am especially indebted to Professor Gang Tian for his invaluable discussions, generous help, consistent support on this problem and drawing my attention to Urbano's article \cite{Urb}. His profound perspectives and academic guidance were crucial in navigating the challenges encountered along the way.
Finally, I wish to thank my friends (or colleagues)-Professors Huyuan Chen, Wenshuai Jiang, Zuoqin Wang, and Qingye Zhang-for their ongoing and constructive discussions, which have greatly enriched my thinking and contributed to the development of this work. The research was supported by the National Natural
Science Foundation of China (Grant No. 12461007) and the Natural Science Foundation of Jiangxi Province (Grant No. 20224BAB201002).
\end{ack}

\end{document}